\documentclass[12pt]{article}
\usepackage{graphicx}
\usepackage[all]{xy}
\usepackage{color}
\usepackage{latexsym}
\usepackage{amssymb,amsmath,amstext}
\usepackage{epsfig}
\usepackage{graphics}
\usepackage{euscript}
\usepackage{ifthen}
\usepackage{wrapfig}
\usepackage{amsthm}
\usepackage{multicol}
\usepackage{tikz}
\usepackage{paralist}
\usepackage{verbatim}
\usepackage[colorlinks=true, urlcolor=blue, linkcolor=blue, citecolor=blue]{hyperref}
\usepackage[margin=1in]{geometry}
\usepackage{algorithmicx}
\usepackage{algpseudocode}
\usepackage{algorithm}
\usepackage{placeins}
\usepackage{caption}
\usepackage{subcaption}
\usepackage{longtable}
\usepackage{lscape}
\usepackage{geometry, afterpage}
\usepackage{fancyhdr}
\usepackage{listings}
\usepackage{breqn}
\usepackage{epstopdf}
\usepackage{kbordermatrix}
\usepackage{bbm}
\usepackage{multicol}

\theoremstyle{plain}
\newtheorem{theorem}{Theorem}[section]
\newtheorem{corollary}[theorem]{Corollary}
\newtheorem{lemma}[theorem]{Lemma}
\newtheorem{proposition}[theorem]{Proposition}

\newtheorem{conjecture}[theorem]{Conjecture}
\theoremstyle{definition}
\newtheorem{definition}[theorem]{Definition}
\newtheorem{remark}[theorem]{Remark}
\newtheorem{example}[theorem]{Example}

\numberwithin{equation}{section}

\newcommand{\Z}{\mathbb{Z}}
\newcommand{\PP}{\mathbb{P}}
\newcommand{\R}{\mathbb{R}}
\newcommand{\Q}{\mathbb{Q}}
\newcommand{\N}{\mathbb{N}}

\newcommand{\sM}{\mathcal{M}}

\newcommand{\CC}{\mathbb C}

\newcommand{\mldeg}{\mathrm{mldeg}}

\DeclareMathOperator{\conv}{conv}

\allowdisplaybreaks

\date{}
\begin{document}
\title{\textbf{Maximum Likelihood Estimation of \\ Toric Fano Varieties}}
\author{Carlos Am\'endola \and Dimitra Kosta \and Kaie Kubjas}

 \maketitle

\begin{abstract}
We study the maximum likelihood estimation problem for several classes of toric Fano models. We start by exploring the maximum likelihood degree for all $2$-dimensional Gorenstein toric Fano varieties. We show that the ML degree is equal to the degree of the surface in every case except for the quintic del Pezzo surface with two ordinary double points and provide explicit expressions that allow one to compute the maximum likelihood estimate in closed form whenever the ML degree is less than 5. We then explore the reasons for the ML degree drop using $A$-discriminants and intersection theory. Finally, we show that toric Fano varieties associated to 3-valent phylogenetic trees have ML degree one and provide a formula for the maximum likelihood estimate. We prove it as a corollary to a more general result about the multiplicativity of ML degrees of codimension zero toric fiber products, and it also follows from a connection to a recent result about staged trees.
\end{abstract}

\section{Introduction}

Maximum likelihood estimation (MLE) is a standard approach to parameter estimation, and a fundamental computational task in statistics. Given observed data and a model of interest, the maximum likelihood estimate is the set of parameters that is most likely to have produced the data. Algebraic techniques have been developed for the computation of maximum likelihood estimates for algebraic statistical models~\cite{agostini2015maximum, allman2019maximum,gross2014maximum,
Hauenstein_etal,Hosten_etal}.
\let\thefootnote\relax\footnote{The authors would like to thank Ivan Cheltsov, Alexander Davie, Martin Helmer, Milena Hering, Steffen Lauritzen, Anna Seigal, Bernd Sturmfels, Hendrik Suess and Seth Sullivant for valuable comments and suggestions. Part of this work was completed, while D. Kosta was supported by a Daphne Jackson Trust Fellowship funded jointly by EPSRC and the University of Edinburgh. C.~Am\'endola was partially supported by the Deutsche Forschungsgemeinschaft (DFG) in the context of the Emmy Noether junior research group KR 4512/1-1. K.~Kubjas was supported by the European Union's Horizon 2020 research and innovation programme: Marie Sk\l{}odowska-Curie grant agreement No. 748354, research carried out at LIDS, MIT and Team PolSys, LIP6, Sorbonne University.}

The \textit{maximum likelihood degree} (ML degree) of an algebraic statistical model is the number of complex critical points of the likelihood function over the Zariski closure of the model~\cite{Catanese_etal}. It measures the complexity of the maximum likelihood estimation problem on a model.  In \cite{Hosten_etal}, an algebraic algorithm is presented for computing all critical points of the likelihood function, with the aim of identifying the local maxima in the probability simplex. In the same article, an explicit formula for the ML degree of a projective variety which is a generic complete intersection is derived and this formula serves as an upper bound for the ML degree of special complete intersections. Moreover, a geometric characterisation of the ML degree of a smooth variety in the case when the divisor corresponding to the rational function is a normal crossings divisor is given in \cite{Catanese_etal}. In the same paper an explicit combinatorial formula for the ML degree of a toric variety is derived by relaxing the restrictive smoothness assumption and allowing some mild singularities. For an introduction to the geometry behind the MLE for algebraic statistical models for discrete data the interested reader is refered to \cite{Huh_Sturmfels}, which includes most of the current results on the MLE problem from the perspective of algebraic geometry as well as statistical motivation.

This article is concerned with the problem of MLE on toric Fano varieties. Toric varieties correpond to log-linear models in statistics. Since the seminal papers by L.A. Goodman in the 70s \cite{Goodman_71, Goodman_73}, log-linear models have been widely used in statistics and areas like natural language processing when analyzing crossclassified data in multidimensional contingency tables \cite{Bishop_etal}. The ML degree of a toric variety is bounded above by its degree. Toric Fano varieties provide several interesting classes of toric varieties for investigating the ML degree drop. We  focus on studying the maximum likelihood estimation for $2$-dimensional Gorenstein toric Fano varieties, the Veronese $(2,2)$ with different scalings and toric Fano varieties associated to 3-valent phylogenetic trees.

Two-dimensional Gorenstein toric Fano varieties correspond to reflexive polygons and by the classification results there are exactly $16$ isomorphism classes of such polygons, see for example~\cite{Kreuzer_Skarke}. Out of these $16$ isomorphism classes five correspond to smooth del Pezzo surfaces and $11$ correspond to del Pezzo surfaces with singularities. Our first main result Theorem~\ref{theorem:2_dim_Gorenstein_toric_fano} states that the ML degree is equal to the degree of the surface in all cases except for the quintic del Pezzo surface with two ordinary double points. Furthermore, in Table \ref{table:closedforms}, we provide explicit expressions that allow the maximum likelihood estimate to be computed in closed form whenever the ML degree is less than five.

We also explore reasons and bounds for the ML degree drop of a toric variety building on the work of Am\'endola et al~\cite{Carlos_etal}.
The critical points of the likelihood function on a toric variety lie in the intersection of the toric variety with a  linear space of complementary dimension.  By B\'ezout's theorem, the sum of degrees of irreducible components of this intersection is bounded above by the degree of the toric variety, and hence the ML degree of a toric variety is bounded by its degree. However, not all the points in the intersection contribute towards the ML degree, i.e. the points with a zero coordinate or the sum of coordinates equal to zero are not counted towards the ML degree. In the case of the quintic del Pezzo surface with two ordinary double points, the ML degree drops by two because there are two points in the intersection of the toric variety and the linear space whose coordinates sum to zero, see Example~\ref{example:ML_degree_drop_5a}. These two points do not depend on the observed data by Lemma~\ref{lemma:removed_points_independent}. Although we do not see this phenomenon with two-dimensional Gorenstein toric Fano varieties,  the ML degree of a toric variety can drop also because the toric variety and the hyperplane intersect nontransversally, and we will see in Sections~\ref{section:ML_degree_drop} and~\ref{section:phylogenetic_models}  that this is often the case.

Buczy\'{n}ska and Wi\'{s}niewski proved that certain varieties associated to 3-valent phylogenetic trees are toric Fano varieties~\cite{Wisniewski_Buczynska_2007}. In phylogenetics, these varieties correspond to the CFN model in the Fourier coordinates. These varieties are examples of codimension zero toric fiber products as defined by Sullivant in~\cite{sullivant2007toric}. Our second main result is Theorem~\ref{theorem:ML_degree_of_toric_fiber_products} that states that the MLE, ML degree as well as critical points of the likelihood function behave multiplicatively in the case of codimension zero toric fiber product of toric ideals. As a corollary, we obtain that the ML degree of the Buczy\'{n}ska-Wi\'{s}niewski phylogenetic variety associated to a 3-valent tree is one and we get a closed form for the MLE. This result holds for the CFN model only in the Fourier coordinates, as the ML degree of the actual model in the probability coordinates  can be much higher.  We observe that the result about the CFN model in the Fourier coordinates can be alternatively deduced from the recent work of Duarte, Marigliano and Sturmfels~\cite{duarte2019discrete}, since Buczy\'{n}ska-Wi\'{s}niewski phylogenetic varieties give staged tree models. It follows from the work of Huh~\cite{Huh} and Duarte, Marigliano and Sturmfels~\cite{duarte2019discrete} that the ML estimator of a variety of ML degree one is given by a Horn map, i.e. an alternating product of linear forms of specific form,  and  such models allow a special characterization using discriminantal triples. We discuss the Horn map and the discriminantal triple  for Buczy\'{n}ska-Wi\'{s}niewski phylogenetic varieties on 3-valent trees in Example~\ref{example:phylogenetic_Huh}.

The outline of this paper is the following. In Section~\ref{section:preliminaries}, we recall preliminaries on maximum likelihood estimation, log-linear models and toric Fano varieties. In Section~\ref{section:del_pezzo}, we study the maximum likelihood estimation for two-dimensional Gorenstein toric Fano varieties. In Section~\ref{section:ML_degree_drop}, we explore the ML degree drop using $A$-discriminants and the intersection theory. Finally, Section~\ref{section:phylogenetic_models} is dedicated to phylogenetic models and codimension zero toric fiber products.

\section{Preliminaries} \label{section:preliminaries}

\subsection{Maximum likelihood estimation}
\label{MLE}

Consider the complex projective space $\mathbb{P}^{n-1}$ with coordinates $(p_1,\ldots,p_n)$. Let $X$ be a discrete random variable taking values on the state space $[n]$. The coordinate $p_i$ represents the probability of the $i$-th event $p_i = P (X = i)$ where $i=1, \ldots, n$. Therefore $p_1+\ldots+p_n =1$. The set of points in $\mathbb{P}^{n-1}$ with positive real coordinates is identified with the probability simplex
$$
\Delta_{n-1} =\{  (p_1, \ldots , p_n) \in \mathbb{R}^{n} :  p_1, \ldots , p_n \geq 0 \text{ and } p_1+\ldots+p_n =1 \} \text{ . }
$$
An algebraic statistical model $\mathcal{M}$ is the intersection of a Zariski closed subset $V \subseteq \mathbb{P}^{n-1}$  with the probability simplex $\Delta_{n-1}$. The data is given by a nonnegative integer vector $(u_1,\ldots,u_n) \in \mathbb{N}^{n}$, where $u_i$ is the number of times the $i$-th event is observed.

The maximum likelihood estimation problem aims to find a model point $p \in \mathcal{M}$ which maximizes the likelihood of observing the data $u$. This amounts to maximizing the corresponding likelihood function
$$
L_u(p_1,\ldots,p_n) = \frac{p_1^{u_1} \cdots p_n^{u_n}}{(p_1+\ldots+p_n)^{(u_1+\ldots+u_n)}}
$$
over the model $\mathcal{M}$. Statistical computations are usually implemented in the affine $n$-plane   $p_1+\ldots+p_n =1$. However, including the denominator makes the likelihood function a well-defined rational function on the projective space  $\mathbb{P}^{n-1}$, enabling one to use projective algebraic geometry to study its restriction to the variety $V$.

he likelihood function might not be concave; it can have many local maxima, making the problem of finding or certifying a global maximum difficult. In algebraic statistics, one tries to find all critical points of the likelihood function, with the aim of identifying all local maxima \cite{Catanese_etal,Hauenstein_etal,Hosten_etal}. This corresponds to solving a system of polynomial equations  called likelihood equations. These equations characterize the critical points of the likelihood function $L_u$. We recall that the number of complex solutions to the likelihood equations, which equals the number of complex critical points of the likelihood function $L_u$ over the variety $V$,  is called the maximum likelihood degree (ML degree) of the variety $V$.

\subsection{Log-linear models}
\label{toric}

In this article we are studying maximum likelihood estimation of log-linear models.  From the algebraic perspective, a log-linear model is a toric variety intersected with a probability simplex, hence log-linear models are sometimes called toric models.
The likelihood function over a log-linear model is concave, although it can have more than one complex critical point over the corresponding toric variety intersected with the plane $p_1+\ldots+p_n=1$ (there is exactly one critical point in the positive orthant). This means that in practice, algorithms like iterative proportional fitting (IPF) are used to find the MLE over a log-linear model. The closed form of the solution is in general not rational and to find its algebraic degree one needs to  compute the ML degree. It is an open problem whether there is a connection between the convergence rate of IPF and the ML degree of a log-linear model~\cite[Section~7.3]{Drton_etal}. The study of the ML degree paired with homotopy continuation methods may speed up the MLE computation with respect to IPF in certain instances, as explored in \cite[Section~8]{Carlos_etal}.

\begin{definition}
Let $A = ( a_{ij} )\in \mathbb{Z}^{(d-1) \times n}$ be an
integer matrix. The log-linear model associated to $A$ is
$$
\mathcal{M}_A=\left\{p \in \Delta_{n-1}:\log p \in \text{rowspan}(A) \right\}.
$$
\end{definition}

Alternatively, a log-linear model can be defined as the intersection of a toric variety and the probability simplex. Recall that $\theta^{a_j} := \theta_1^{a_{1,j}}  \theta_2^{a_{2,j}} \cdots \theta_d^{a_{d-1,j}}$ for $j=1,\ldots,n$.

\begin{definition} \label{def:toric_variety}
Let $A = ( a_{ij} ) \in \mathbb{Z}^{(d-1) \times n}$ be an
integer  matrix. The toric variety $V_A \subseteq \mathbb{R}^{n}$ is the Zariski closure of the image of the parametrization map
$$
\psi:(\mathbb{C}^*)^d \rightarrow (\mathbb{C}^*)^{n},\left(s,\theta_1,\ldots,\theta_{d-1} \right) \mapsto \left(s \theta^{a_1},\ldots,s\theta^{a_{n}} \right).
$$
The ideal of $V_A$ is denoted by $I_A$ and called the toric ideal associated to $A$.
\end{definition}

Often the columns of $A$  are lattice points of a lattice polytope $Q \subseteq \mathbb{R}^{d-1}$. In this case we say that $V_A$ is the toric variety corresponding to $Q$.
The log-linear model $\mathcal{M}_A$ is the intersection of the toric variety $V_A$ with the probability simplex $\Delta_{n-1}$. We omit $A$ from the notation whenever it is clear from the context. We conclude this subsection with a characterization of the MLE for log-linear models.

\begin{proposition}[Corollary 7.3.9 in~\cite{sullivant2018algebraic}] \label{Birch}
Let $A$ be a $(d-1) \times n$ nonnegative
integer  matrix  and let $u \in \mathbb{N}^{n}$ be a data vector of size
$u_+ = u_1 + \cdots+ u_n$. The maximum likelihood estimate over the model $\mathcal{M}_A$ for the data $u$ is the unique solution $\hat{p}$, if it exists, to
$$
\hat{p}_1+\ldots + \hat{p}_n=1, \quad A\hat{p} = \frac{1}{u_+} Au \quad  \text{and} \quad \hat{p} \in \mathcal{M}_A.
$$
\end{proposition}

Proposition~\ref{Birch} is also known as Birch's Theorem. Often we consider $V_A$ as a projective variety in $\mathbb{P}^{n-1}$. The projective version of Proposition~\ref{Birch} is given in Section~\ref{section:ML_degree_drop}. We usually use the affine version when we want to compute the ML degree or find critical points of the likelihood function and the projective version when studying the ML degree drop.

\subsection{Toric Fano varieties}
\label{delPezzo}
In this section we will provide a brief introduction to toric Fano varieties, the main objects of study in this article. Fano varieties are a class of varieties with a special positive divisor class giving an embedding of each variety into projective space. They were introduced by Giro Fano \cite{Fano} and have been extensively studied in birational geometry in the context of the minimal model program (see \cite{Kollar}, \cite{IskovskikhProkhorov}).

\begin{definition}
A complex projective algebraic variety $X$ with ample anticanonical divisor class $-K_X$ is called a \emph{Fano variety}.
\end{definition}

Two-dimensional Fano varieties are also known as \emph{del Pezzo surfaces} named after the Italian mathematician Pasquale del Pezzo who encountered this class of varieties when studying surfaces of degree $d$ embedded in $\mathbb{P}^d$. Throughout this paper we will use the terminology del Pezzo surface to refer to a two-dimensional Fano variety. We note that we do not use the terminology Fano surface, as a Fano surface usually refers to a surface of general type whose points index the lines on a non-singular cubic threefold, which is not a Fano variety \cite{Fano04}.

We will consider Fano varieties that are also toric varieties as defined in Definition~\ref{def:toric_variety}. We first focus on the characterization of two-dimensional Gorenstein toric Fano varieties, i.e. normal toric Fano varieties whose anticanonical divisor $K_X$  is not only an ample divisor but also a Cartier divisor. Isomorphism classes of Gorenstein toric Fano varieties are in bijection with isomorphism classes of reflexive polytopes, which were introduced in \cite{Batyrev}.

\begin{definition}
A lattice polytope is reflexive if it contains the origin in its interior and its dual polytope is also a lattice polytope.
\end{definition} 

In particular, toric del Pezzo surfaces are in bijection with two-dimensional reflexive polytopes. The classification of two-dimensional reflexive polytopes can be found for example in~\cite{Kreuzer_Skarke}.

\begin{proposition}[Section 4 in~\cite{Kreuzer_Skarke}] There are exactly 16 isomorphism classes of two-dimensional reflexive polytopes depicted in Figure~\ref{table:polytopes}.
\end{proposition}

\begin{figure}[ht]
\centering
 \includegraphics[scale=0.80]{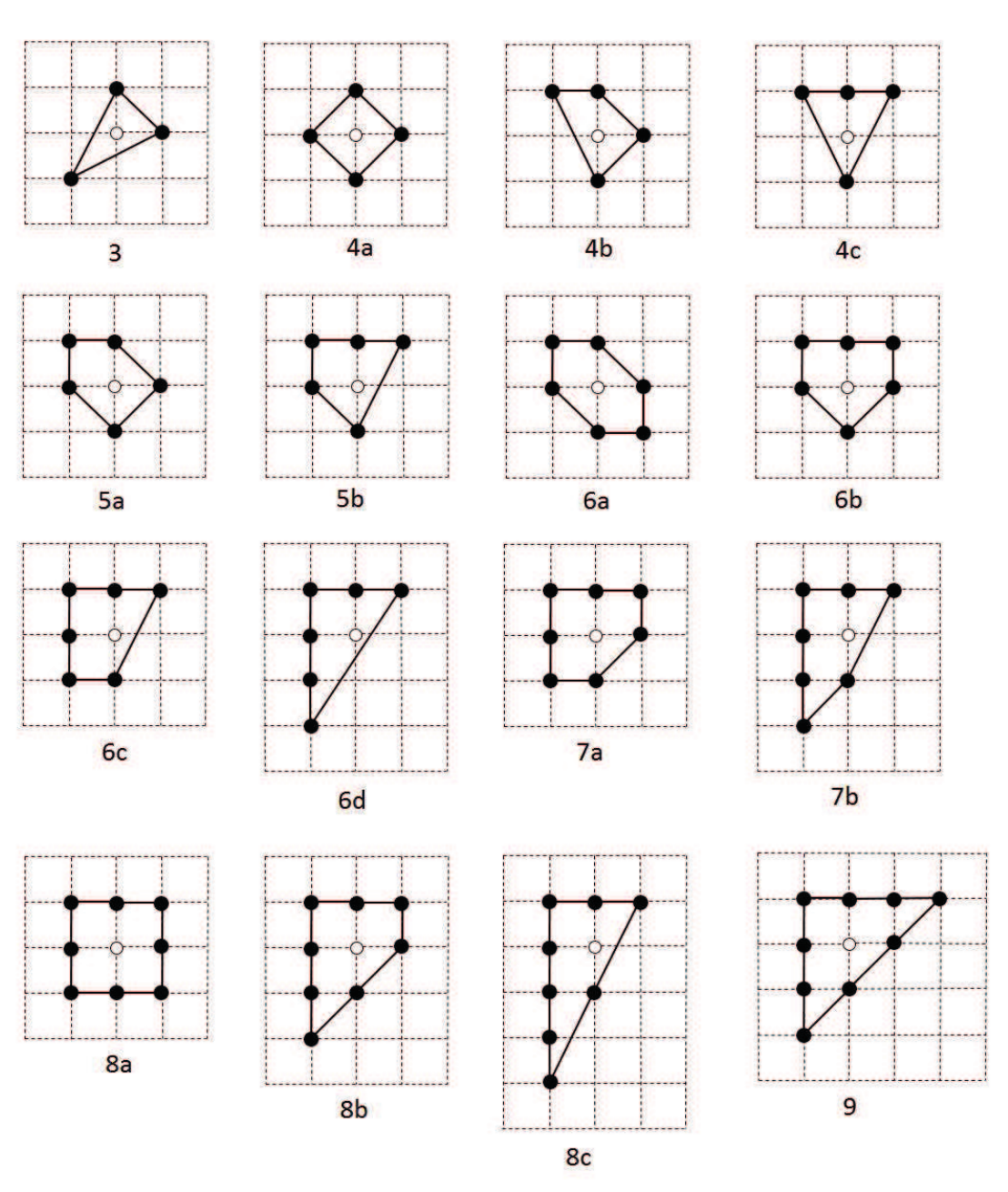}
 \caption{Isomorphism classes of two-dimensional reflexive polytopes}
\label{table:polytopes}
\end{figure}

  In Figure~\ref{table:polytopes}, the reflexive polytopes are labeled by the number of lattice points on the boundary. On the other hand, the self-intersection number $K^2_S$ of the canonical class of a del Pezzo surface is the \emph{degree} of the del Pezzo surface, which we denote by $d$. Here we adopt the approach of \cite[Chapter 8.3]{cox2011toric}, where the reflexive polytope is the one corresponding to the anticanonical embedding of the del Pezzo surface. According to \cite[Chapter 8.3, Ex.  8.3.8]{cox2011toric}, for each of the 16 reflexive polytopes we obtain exactly the corresponding toric del Pezzo surface. Furthermore, in \cite[Chapter 8]{Dolgachev} the degree of each of these surfaces is given and coincides with the number of lattice points on the boundary. In this way, the projective varieties corresponding to the polytopes labeled by $6a,7a,8a,8b$ and $9$ are smooth and the projective varieties corresponding to the rest of the polytopes in~Figure~\ref{table:polytopes} have singularities. The dual of the polytope labeled by number $x$ and letter $y$ is in the isomorphism class of the polytope labeled by number $12-x$ and letter $y$. This is related to the so-called ``12 theorem" for reflexive polytopes of dimension 2 \cite{godinho2017}.
  
\begin{remark}
As explained above, in the manner that toric varieties were defined in Definition~\ref{def:toric_variety}, the degree of the toric variety corresponding to a polytope $Q$ and the number of lattice points on the boundary of $Q$ coincide. However, sometimes in the literature (see for instance \cite[Example, p.~123]{Casagrande}) the dual polytope is used to characterize the isomorphism class of a toric del Pezzo surface. In our setting, the corresponding polytope for $\mathbb{P}^2$ is the polytope $9$ in Figure~\ref{table:polytopes} which gives the anticanonical embedding, i.e. the degree $3$ Veronese embedding into $\mathbb{P}^9$ using the linear system of cubics.
\end{remark}

\section{MLE of two-dimensional Gorenstein toric Fanos} \label{section:del_pezzo}
In this section we determine the ML degree of two-dimensional Gorenstein toric Fano varieties. When the ML degree is less than or equal to three, we reduce the likelihood equations to relatively simple expressions that can be used to compute a closed form for the maximum likelihood estimates. We use the cubic del Pezzo surface as an example to illustrate the MLE derivation. To avoid statistical difficulties, in all of this section we have translated reflexive polygons by a positive vector such that the resulting polygons lie minimally in the positive orthant.
\begin{theorem} \label{theorem:2_dim_Gorenstein_toric_fano}
Let $S_d$ be a two-dimensional Gorenstein toric Fano variety. In Table~\ref{table:MLdegs} we determine the ML degree of $S_d$ and show that it is equal to the degree $d$ of the surface in all cases except for the quintic surface $S_{5a}$. Table~\ref{table:closedforms} provides explicit expressions for computing the maximum likelihood estimate of the algebraic statistical models corresponding to the cubic $S_3$, the quartics $S_{4a}, S_{4b}, S_{4c}$ and the quintic $S_{5a}$ toric two-dimensional Fano variety.
\end{theorem}

\begin{table}[htp]
\centering
\begin{tabular}{||c c||}
 \hline
 Ideal of some degree $d$ del Pezzo $S_d$ & $\mldeg$\\ [0.5ex]
 \hline\hline
 $S_3: p_1 p_2 p_3 - p_4^3$ & 3  \\
 \hline
 $S_{4a}: p_2p_4-p_1p_5, p_3^2-p_1p_5$ & 4 \\
 \hline
 $S_{4b}: p_2p_4-p_3^2, p_2p_3-p_1p_5$ & 4 \\
 \hline
 $S_{4c}: p_2p_4-p_3^2, p_2^2-p_1p_5$ & 4 \\
 \hline
$S_{5a}: p_3p_5-p_4p_6, p_2p_5-p_6^2, p_2p_4-p_3p_6, p_1p_4-p_6^2, p_1p_3-p_2p_6$ & 3\\
\hline
$S_{5b}: p_3p_5-p_4p_6, p_2p_5-p_6^2, p_2p_4-p_3p_6, p_1p_4-p_2p_6, p_2^2-p_1p_3$ & 5 \\
\hline
$S_{6a}: p_4p_6-p_5p_7, p_3p_6-p_7^2, p_2p_6-p_1p_7, p_3p_5-p_4p_7, p_2p_5-p_7^2$ & \\
$p_1p_5-p_6p_7, p_2p_4-p_3p_7, p_1p_4-p_7^2, p_1p_3-p_2p_7$ & 6 \\
\hline
$S_{6b}: p_5 p_6-p_1 p_7, p_4 p_6-p_2 p_7, p_3 p_5-p_4 p_7, p_2 p_5-p_7^2, p_2 p_4-p_3 p_7$ & \\
$p_1 p_4-p_7^2, p_1 p_3-p_2 p_7, p_2^2-p_3 p_6, p_1 p_2-p_6 p_7$ & 6 \\
\hline
$S_{6c}: p_6^2-p_5 p_7, p_4 p_6-p_3 p_7, p_3 p_6-p_2 p_7, p_4 p_5-p_2 p_7, p_3 p_5-p_2 p_6$ & \\
$p_2 p_4-p_1 p_7, p_3^2-p_1 p_7, p_2 p_3-p_1 p_6, p_2^2-p_1 p_5$ & 6 \\
\hline
$S_{6d}: p_6^2-p_5p_7, p_5p_6-p_4p_7, p_3p_6-p_2p_7, p_5^2-p_4p_6, p_3p_5-p_2p_6$ & \\
$p_3p_4-p_2p_5, p_3^2-p_1 p_6, p_2 p_3-p_1 p_5, p_2^2-p_1p_4$ & 6 \\ [1ex]
\hline
$S_{7a}: p_5p_7-p_4p_8, p_4p_7-p_3p_8, p_2p_7-p_1p_8, p_5p_6-p_3p_8, p_4p_6-p_3p_7$ & \\
$p_2p_6-p_1p_7, p_4p_5-p_2p_8, p_3p_5-p_1p_8, p_4^2-p_1p_8, p_3p_4-p_1p_7$ & \\
$p_2p_4-p_1p_5, p_3^2-p_1p_6, p_7^2-p_6p_8, p_2p_3-p_1p_4$ & 7 \\ [1ex]
\hline
$S_{7b}: p_7^2-p_6p_8, p_6p_7-p_5p_8, p_4p_7-p_3p_8, p_3p_7-p_2p_8, p_6^2-p_5p_7$ & \\
$p_4p_6-p_2p_8, p_3p_6-p_2p_7, p_4p_5-p_2p_7, p_3p_5-p_2p_6, p_3p_4-p_1p_8$ & \\
$p_2p_4-p_1p_7,p_3^2-p_1p_7, p_2p_3-p_1p_6,p_2^2-p_1p_5$ & 7 \\ [1ex]
\hline
$S_{8a}: p_8^2-p_7p_9, p_6p_8-p_5p_9, p_5p_8-p_4p_9, p_3p_8-p_2p_9, p_2p_8-p_1p_9$ & \\
$p_6p_7-p-4p_9, p_5p_7-p_4p_8, p_3p_7-p_1p_9, p_2p_7-p_1p_8, p_6^2-p_3p_9$ & \\
$p_5p_6-p_2p_9, p_4p_6-p_1p_9, p_5^2-p_1p_9, p_4p_5-p_1p_8,p_3p_5-p_2p_6$ & \\
$p_2p_5-p_1p_6, p_4^2-p_1p_7, p_3p_4-p_1p_6, p_2p_4-p_1p_5,p_2^2-p_1p_3$ & 8 \\ [1ex]
\hline
$S_{8b}: p_8^2-p_7p_9, p_7p_8-p_6p_9, p_5p_8-p_4p_9, p_4p_8-p_3p_9, p_2p_8-p_1p_9$ & \\
$p_7^2-p_6p_8, p_5p_7-p_3p_9, p_4p_7-p_3p_8, p_2p_7-p_1p_8, p_5p_6-p_3p_8$ & \\
$p_4p_6-p_3p_7, p_2p_6-p_1p_7, p_5^2-p_2p_9, p_4p_5-p_1p_9, p_3p_5-p_1p_8$ & \\
$p_4^2-p_1p_8, p_3p_4-p_1p_7, p_2p_4-p_1p_5, p_3^2-p_1p_6, p_2p_3-p_1p_4$ & 8 \\ [1ex]
\hline
$S_{8c}: p_8^2-p_7p_9, p_7p_8-p_6p_9, p_6p_8-p_5p_9, p_4p_8-p_3p_9, p_3p_8-p_2p_9$ & \\
$p_7^2-p_5p_9, p_6p_7-p_5p_8, p_4p_7-p_2p_9, p_3p_7-p_2p_8, p_6^2-p_5p_7$ & \\
$p_4p_6-p_2p_8, p_3p_6-p_2p_7, p_4p_5-p_2p_7, p_3p_5-p_2p_6, p_4^2-p_1p_9$ & \\
$p_3p_4-p_1p_8, p_2p_4-p_1p_7, p_3^2-p_1p_7, p_2p_3-p_1p_6,p_2^2-p_1p_5$ & 8 \\ [1ex]
\hline
$S_{9}: p_9^2-p_8p_{10}, p_8p_9-p_7p_{10},p_6p_9-p_5p_{10},p_5p_9-p_4p_{10}, p_3p_9-p_2p_{10}$ & \\
$p_8^2-p_7p_9,p_6p_8-p_4p_{10},p_5p_8-p_4p_9, p_3p_8-p_2p_9, p_6p_7-p_4p_9$ & \\
$p_5p_7-p_4p_8, p_3p_7-p_2p_8,p_6^2-p_3p_{10}, p_5p_6-p_2p_{10}, p_4p_6-p_2p_9$ & \\
$p_3p_6-p_1p_{10},p_2p_6-p_1p_9, p_5^2-p_2p_9, p_4p_5-p_2p_8, p_3p_5-p_1p_9$ & \\
$p_2p_5-p_1p_8, p_4^2-p_2p_7, p_3p_4-p_1p_8, p_2p_4-p_1p_7,p_3^2-p_1p_6$ & \\
$p_2p_3-p_1p_5, p_2^2-p_1p_4$ & 9 \\ [1ex]
\hline
\end{tabular}
\caption{ML degrees of 2-dimensional Gorenstein toric Fanos} \label{table:MLdegs}
\end{table}

Table \ref{table:MLdegs} is constructed using Theorem~\ref{Birch} and \texttt{Macaulay2} \cite{M2}. The results described in Table~\ref{table:MLdegs} are in accordance with \cite[Theorem 3.2]{Huh_Sturmfels}, which states that the ML degree of a projective toric variety is bounded above by its degree.  We see in Table~\ref{table:MLdegs} that the ML degree drops to three in the case of a quintic del Pezzo surface $S_{5a}$ corresponding to the reflexive polytope $5a$ in  Table~\ref{table:polytopes}. The next section provides an explanation of the ML degree drop in the case of the quintic $S_{5a}$ using the notion of $A$-discriminant.

\begin{example}[Singular Cubic del Pezzo surface]
\label{example:cubic}
Consider the reflexive polytope
\begin{center}
 \includegraphics[scale=0.35]{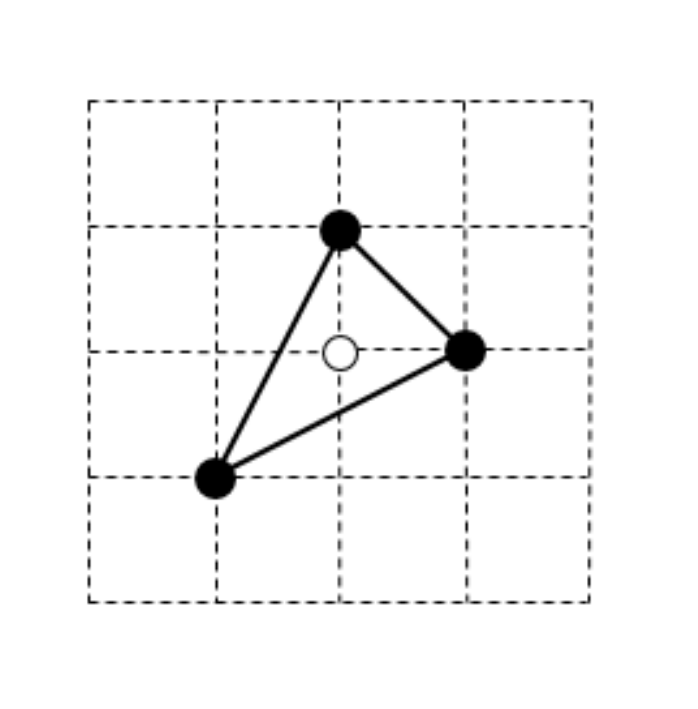}
\end{center}
The corresponding projective toric variety is a cubic surface $S_3$ in $\mathbb{P}^3$ with three singular points. Its ideal is generated by $I_{S_3}= < p_4^3 - p_1 p_2 p_3 >$.

We are interested in the algebraic statistical model given by the matrix
\[
A=
  \begin{bmatrix}
    2 & 1 & 0 & 1\\
    1 & 2 & 0 & 1

  \end{bmatrix}.
\]
This nonnegative integer matrix $A$ gives the parametrization map
$$f: \mathbb{C}^3 \to \mathbb{C}^3  \text{  ,  }   (s, \theta_1, \theta_2) \mapsto ( s\theta_1^2 \theta_2 , s\theta_1 \theta_2^2 , s, s\theta_1 \theta_2  ) .$$

After applying Birch's theorem, we can write the unique maximum likelihood estimate $\hat{s},\hat{\theta}$ for the data $u$  as
$(\hat{s},\hat{\theta}_1, \hat{\theta}_2) = (\frac{\hat{p}_4^3}{\hat{p}_1 \hat{p}_2},\frac{\hat{p}_1}{\hat{p}_4}, \frac{\hat{p}_2}{\hat{p}_4})$, where
\begin{eqnarray*}
\hat{p}_1 & = & x+a,\\
\hat{p}_2 & = & x+b,\\
\hat{p}_4 & = & -3x+c,
\end{eqnarray*}
with $a=\frac{u_1-u_3}{u_+}, b=\frac{u_2-u_3}{u_+}, c=\frac{3u_3+u_4}{u_+} $
and $x$ is given by

 \begin{dmath*}
 28x^3 + [(a+b)-27c] x^2 + [ab+9c^2]x  - c^3 = 0.
\end{dmath*}
\end{example}

\begin{landscape}
\begin{table}[htbp]
\begin{tabular}{|c|c|c|}
\hline
 Polytope $Q_d$ & MLE & Polynomial equation of degree $d=\mldeg$  \\
\hline
 & $\hat{s} = \frac{(-3x+c)^3}{(x+a)(x+b)}$  & $ 28x^3 + [a+b-27c] x^2 + [ab+9c^2]x  - c^3 = 0$\\
 \parbox[c]{1em}{\includegraphics[scale=0.3]{LatticePolytope_3pts_1.pdf}} \hspace{1cm} & $\hat{\theta}_1 = \frac{x+a}{-3x+c}$   & where \\
 & $\hat{\theta}_2 = \frac{x+b}{-3x+c}$ & where $a=\frac{u_1-u_3}{u_+}, b=\frac{u_2-u_3}{u_+}, c=\frac{3u_3+u_4}{u_+} $\\
\hline
 & $\hat{s} = \frac{(-x^2+8x+8-(3b^2-4a^2+4))^3}{16(x-1)^2(x+b)(-7x^2+(8a+8)x+(3b^2-4a^2-4-8a))}$ & $ 15x^4-16x^3+(8a^2-22b^2-56)x^2+(16(4-4a^2+5b^2))x$ \\
 \parbox[c]{1em}{\includegraphics[scale=0.3]{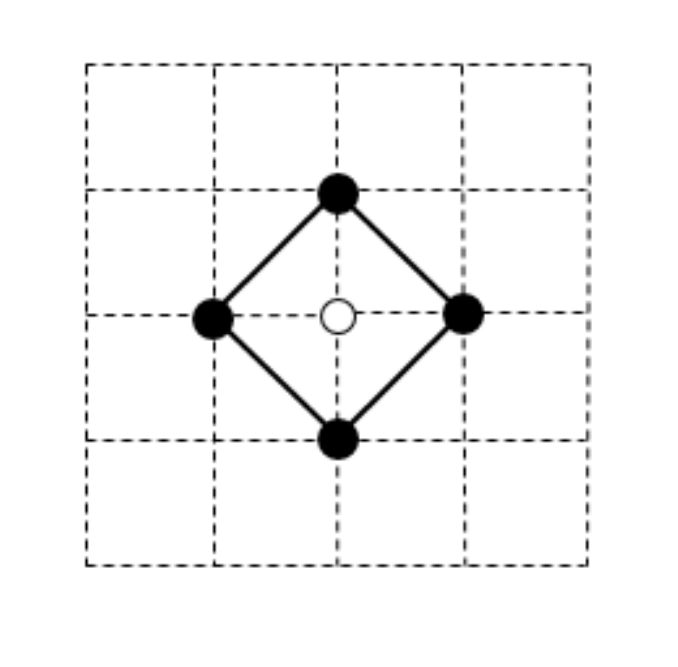}} \hspace{1cm} & $\hat{\theta}_1 = \frac{-7x^2+(8a+8)x+(3b^2-4a^2-4-8a)}{2(-x^2+8x-(3b^2-4a^2+4))}$ & $+8(4a^2-5b^2-2)-(4a^2-3b^2)^2 = 0 $\\
 &  &  where $a=\frac{u_1-u_4}{u_+}, b=\frac{u_2-u_3}{u_+}, c=\frac{2u_3+2u_4+u_5}{u_+}$ \\
& $\hat{\theta}_2 = \frac{4(x+b)(x-1)}{-x^2+8x-(3b^2-4a^2+4)}$ & \\
\hline
 & $\hat{s} = \frac{(4x^2+2(c-1)x+ab)^3}{(1-x)(x^2+(c+1)x+ab+c)(-2x^2-(2c+a)x+a-ab)}$ & $17x^4+(17c-16)x^3 +(3+9ab-8c+4c^2)x^2$ \\
 \parbox[c]{1em}{\includegraphics[scale=0.3]{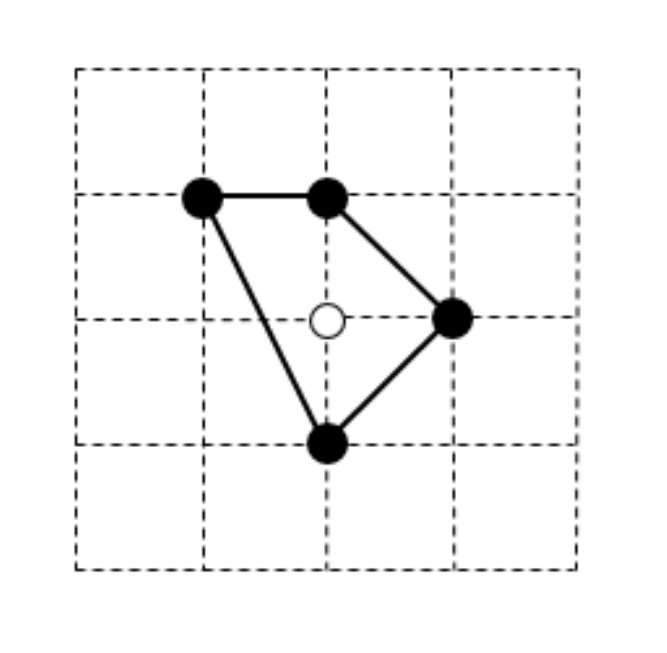}} \hspace{1cm} & $\hat{\theta}_1 = \frac{-2x^2-(a+2c)x+(a-ab)}{4x^2+2(c-1)x+ab}$ & $+(4 a b c-5ab - c)x + a^2b^2=0$\\
& $\hat{\theta}_2 = \frac{x^2+(c+1)x+(ab+c)}{4x^2+2(c-1)x+ab}$ & where $a=\frac{u_1+2u_4+u_5}{u_+}$, $b=\frac{u_3+2u_4+u_5}{u_+}$ and $c=\frac{u_2-3u_4-u_5}{u_+}$\\
\hline
 & $\hat{s} = \frac{4x^2(b-x)}{-3x^2-(2a+2)x+(4a+c)b}$ & $-55x^4 +12x^3 +(c(4a+c)+b(5b-8))x^2$ \\
 \parbox[c]{1em}{\includegraphics[scale=0.3]{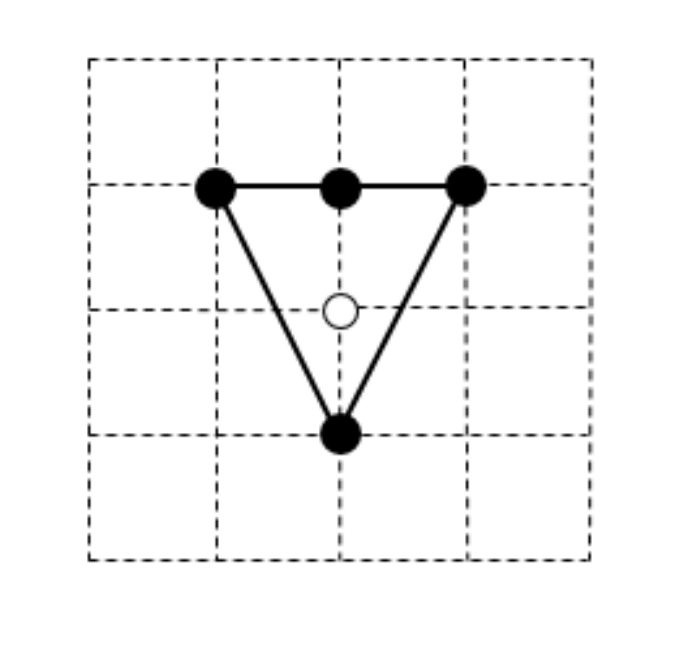}} \hspace{1cm} & $\hat{\theta}_1 = \frac{-3x^2-(2a+2)x+(4a+c)b}{8x^2}$ & $ -(4b(ab+ac+c))x + (4a+c)b^2c  = 0$ \\
& $\hat{\theta}_2 = \frac{2x}{b-x}$ & where $a=\frac{u_2-u_4}{u_+}$, $b=\frac{2u_1+u_5}{u_+}$ and $c=\frac{2u_3+4u_4+u_5}{u_+}$\\
\hline
 & $\hat{s} = \frac{(-x^2+cx)(x^2+(a+b)x)}{2x^2-(b+2c)x+bc}$ & $ -5x^3 + (3-5a)x^2 + (-a-b(b+5c)) x + b^2c = 0$ \\
  \parbox[c]{1em}{\includegraphics[scale=0.34]{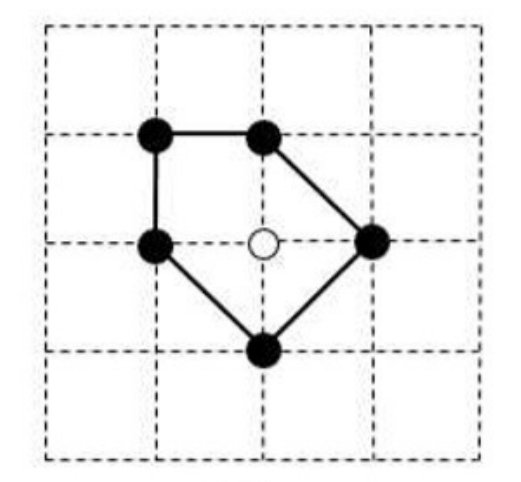}} \hspace{0.7cm} & $\hat{\theta}_1 = \frac{2x^2-(b+2c)x+bc}{x^2+(a+b)x} $ &  where \\
& $\hat{\theta}_2 = \frac{2x^2-(b+2c)x+bc}{-x^2+cx}$ & $a=\frac{u_2-u_3-3u_4-2u_6}{u_+}$, $b=\frac{u_3+u_5+2(u_4+u_6)}{u_+}$ and $c=\frac{u_1+u_3+u_6+2u_4}{u_+}$\\
\hline
\end{tabular}
\caption{Explicit forms for the MLE for 2-dimensional Gorenstein toric Fanos}  \label{table:closedforms}
\end{table}
\end{landscape}

\begin{remark} When the ML degree of the del Pezzo surface is greater than or equal to five, the maximum likelihood estimate  $\hat{p}_i , \, i=1,\ldots, n$ satisfies an equation of degree five or higher. By the Abel-Ruffini theorem there is no algebraic solution for a general polynomial equation of degree five or higher, therefore one would expect that it is not possible to obtain a closed form solution for the maximum likelihood estimate in these cases. However, one can then turn to numerical algebraic geometry methods to compute the MLE (see e.g. \cite{Hauenstein_etal}).
\end{remark}

\section{ML degree drop} \label{section:ML_degree_drop}

In order to understand why the ML degree is lower than the degree for the quintic del Pezzo surface $5a$, it is useful to think of different embeddings of a toric variety via scalings and how these affect the ML degree. For a full analysis see \cite{Carlos_etal}.

Let $Q \subseteq \mathbb{R}^{d-1}$ be a lattice polytope with $n$ lattice points $a_j \in \mathbb{Z}^{d-1}$. Define $A$ to be the $(d-1) \times n$ matrix with the  columns $a_1,\ldots,a_n$. A scaling $c \in (\CC^*)^n$ can be used to define the parametrization
$\psi^c \, : \, (\CC^*)^d \longrightarrow (\CC^*)^n$ as
$$ \psi^c(s,\theta_{1},\dots,\theta_{d-1})  \, = \, (c_{1}s\theta^{a_{1}},\dots,c_{n}s\theta^{a_{n}}).$$
We denote by $V^c$ the Zariski closure of the image of the monomial map $\psi^c$. The usual parametrization of the toric variety is when $c=(1,\dots,1)$. We then denote the corresponding toric variety by $V= V^{(1,\dots,1)}$. 

\begin{definition}
The \emph{ML degree drop} of a scaled toric variety $V^c$ is the difference $ \deg(V) - \mldeg(V^c)$.
\end{definition}

Define $f_c = \sum_{i=1}^n c_i \theta^{a_i}$ where $c = (c_1, \ldots, c_n) \in (\CC^*)^n$.

\begin{definition}\label{def:Adisc}
To any matrix $A$ as above, one can associate the variety
\[
\nabla_{A} =\overline{\left\{ c\in\left(\mathbb{C}^{*}\right)^{n} \mid\exists\theta\in\left(\mathbb{C}^{*}\right)^{d-1}\text{ such that }f_c\left(\theta\right)=\frac{\partial f_c}{\partial\theta_{i}}\left(\theta\right)=0\text{ for all }i\right\} }.
\]
This is the Zariski closure of the set of scalings $c \in \left( \mathbb{C}^{*}\right)^{n}$ such that the hypersurface $\{ f_c=0 \}$ has a singular point in $\left(\mathbb{C}^{*}\right)^{d-1}$. If the affine lattice generated by $A$ is the full integer lattice, that is $\{\sum \lambda_j a_j: \lambda_j \in \mathbb{Z}, \sum \lambda_j=1\}=\mathbb{Z}^{d-1}$, then the variety $\nabla_{A}$ is a hypersurface. In this case, it is defined by an irreducible polynomial denoted $\Delta_{A}$, called the \emph{$A$-discriminant}~\cite[Chapter~8]{GKZ}.
\end{definition}

The main object that determines whether the ML degree drops is the polynomial:
\begin{equation}\label{eqn:Adet}
E_A({c}) \, = \, \prod_{\Gamma \text{ face of } Q} \Delta_{\Gamma \cap A}(c)
\end{equation}
 where the product is taken over all nonempty faces $\Gamma \subset Q$ including $Q$ itself and $\Gamma \cap A$ is the matrix whose columns correspond to the lattice points contained in $\Gamma$. Under certain conditions this is precisely the \emph{principal A-determinant} \cite[Chapter~10]{GKZ}.

\begin{theorem}[Theorem 2 in~\cite{Carlos_etal}] \label{theorem:ML_degree_drop}
Let $V^c \subset \PP^{n-1}$ be the scaled toric variety defined by the monomial parametrization with scaling $c \in (\CC^*)^n$ fixed. Then $\mldeg(V^c) < \deg(V)$ if and only if $E_A(c)=0$.
\end{theorem}

\begin{example} \label{example:ML_degree_drop_5a}
We will explain why for $c=(1,1,\dots,1)$, the ML degree of the quintic del Pezzo $5b$ is 5 (and thus equal to its degree), while the ML degree of the quintic del Pezzo $5a$ is strictly less than 5.

Let us consider first the case of the quintic del Pezzo $5b$.
\begin{center}
 \includegraphics[scale=0.50]{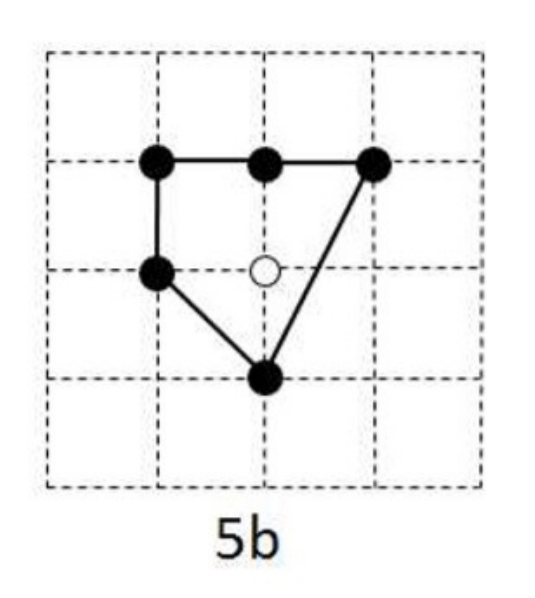}
\end{center}
We can label its lattice points and arrange them in the matrix
$$A = \begin{bmatrix}
0 & 0 & 1 & 1 & 1 & 2 \\
1 & 2 & 0 & 1 & 2 & 2
\end{bmatrix} $$
We have to check that for $c = (1,\dots, 1)$, $E_A(c) \neq 0 $. By \eqref{eqn:Adet}, the polynomial $E_A(c)$ is a product of $\Gamma \cap A$-discriminants. Vertices $a_i$ have $\Delta_{a_i}=1$.  Analogously, edges of lattice length one cannot have non-trivial discriminant (the lattice length of an edge is the number of lattice points contained in the edge minus one). The only potential edge $e$ that may be relevant here is the one of lattice length 2. The corresponding $f_{c,e}= c_{02}y^2 + c_{12}xy^2 + c_{22}x^2y^2 $ has a nontrivial singularity if and only if $c_{02} + c_{12}x + c_{22}x^2 $ does, thus $\Delta_e(c) = c_{12}^2 - 4c_{02}c_{22}$. Note it is non-zero for $c=(1,\dots,1)$.

It only remains to check that $(1,\dots,1) \notin \nabla_A$. The following \texttt{M2} computation verifies that for $c=(1,\dots,1)$,  $f_c=y+y^2+x+xy^2+x^2y^2+xy$ has no singularities:

\begin{verbatim}
R = QQ[x,y]
J = ideal(y+y^2+x+x*y^2+x^2*y^2+x*y, 1+y^2+2*x*y^2+y,
1+2*y+2*x*y+2*x^2*y+x)
gens gb J
\end{verbatim}
The last command returns that the Gr\"obner basis for $J$ is $\{ 1 \}$.
Now, for the quintic del Pezzo $5a$,
\begin{center}
 \includegraphics[scale=0.50]{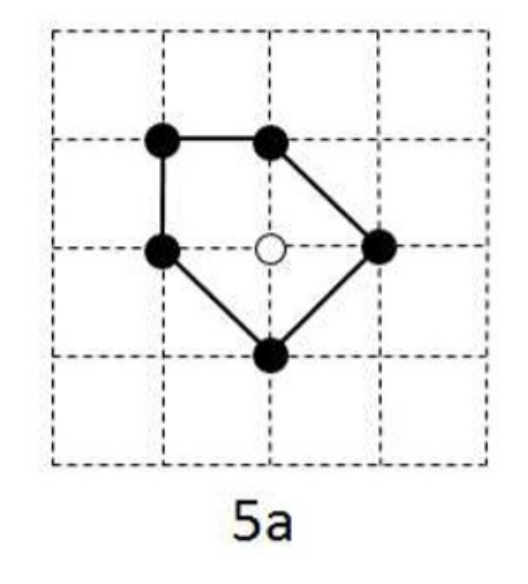}
\end{center}
we identify the matrix $$A = \begin{bmatrix}
0 & 0 & 1 & 1 & 1 & 2 \\
1 & 2 & 0 & 1 & 2 & 1
\end{bmatrix}.  $$  All edges are of lattice length one so we again focus on $\nabla_A$. However, now $c=(1,\dots,1) \in \nabla_A$, as the following code verifies.
 \begin{verbatim}
 I = ideal(y+y^2+x+x*y^2+x^2*y+x*y, 1+y^2+2*x*y+y, 1+2*y+2*x*y+x^2+x)
gens gb I
\end{verbatim}
In this case we get 2 points, the solutions of $x+y=0, y^2-y-1=0$, as singularities for $fc=y+y^2+x+xy^2+x^2y+xy$. The corresponding points of the variety are
\begin{equation}\label{equation:5a_points_that_cause_ML_degree_drop}
\begin{aligned}
(1/2(1+\sqrt{5}),1/2(3+\sqrt{5}),-1/2(1+\sqrt{5}),-1/2(3+\sqrt{5}),-2-\sqrt{5},2+\sqrt{5}),\\
(1/2(1-\sqrt{5}),1/2(3-\sqrt{5}),-1/2(1-\sqrt{5}),-1/2(3-\sqrt{5}),-2+\sqrt{5},2-\sqrt{5}).
\end{aligned}
\end{equation}
 According to Theorem \ref{theorem:ML_degree_drop}, the ML degree must drop for $5a$.
\end{example}

\begin{remark}
The singular locus of the quintic del Pezzo $S_{5a}$ consists of the two distinct points $(0,0,1,0,0,0)$ and $(0,0,0,0,0,1)$ which are both rational double points. These points are different from the two points~(\ref{equation:5a_points_that_cause_ML_degree_drop}) that cause the ML degree drop.
\end{remark}

Theorem~\ref{theorem:ML_degree_drop} characterizes scaling factors $c$ such that the ML degree of $V^c$ is less than the degree of $V$. All critical points of the likelihood function of $V^c$ lie in the intersection of $V$ with a linear space. In the rest of this section, we will investigate the ML degree drop for a given toric variety $V^c$ by studying this intersection.

Let $L_c(p)=\sum_{i=1}^n c_i p_i$ and $L_{c,i}(p)=\sum_{j=1}^n A_{ij} c_j p_j$ for $i=1,\ldots,d-1$. These polynomials are implicit versions of the polynomials $f_c$ and $\theta_i \frac{\partial f_c}{\partial \theta_i}$ for $i=1,\ldots,d-1$.
By~\cite[Proposition 7]{Carlos_etal} the ML degree of $V^c$ is the number of points $p$ in $V \backslash \mathcal{V}(p_1\cdots p_n(c_1 p_1+ \ldots + c_n p_n))$ satisfying
\begin{equation} \label{equations:linear_subspace}
(Au)_i L_c(p) = u_+ L_{c,i}(p) \quad \text{for } i=1,\ldots,d-1
\end{equation}
for generic vectors $u$.
Define $\mathcal{L}'_{c,u}$ to be the intersection of $V$ with the solution set of~(\ref{equations:linear_subspace}) and $\mathcal{L}_{c,u}$ to be the intersection of $V \backslash \mathcal{V}(p_1\cdots p_n(c_1 p_1+ \ldots + c_n p_n))$ with the solution set of~(\ref{equations:linear_subspace}).  By~\cite[Example 12.3.1]{fulton1998intersection}, the sum of degrees of the irreducible components of $\mathcal{L}'_{c,u}$ is at most $\deg V$.

The obvious reason for the  ML degree drop comes from removing these irreducible components of $\mathcal{L}'_{c,u}$ that belong to $\mathcal{V}(p_1\cdots p_n(c_1 p_1+ \ldots + c_n p_n))$. We will see in Lemma~\ref{lemma:removed_points_independent} that the irreducible components of $\mathcal{L}'_{c,u}$ that are removed do not depend on $u$  but only on $c$ and the variety $V$. In the case of the toric del Pezzo surface 5a, the ML degree drop is completely explained by this reason. The degree of this del Pezzo surface is five. The variety $\mathcal{L}'_{c,u}$ consists of two zero-dimensional components of degrees three and two. The degree two component consists of two points~(\ref{equation:5a_points_that_cause_ML_degree_drop}) that lie in the variety $\mathcal{V}(p_1\cdots p_n(c_1 p_1+ \ldots + c_n p_n))$ and hence is removed.

\begin{lemma} \label{lemma:removed_points_independent}
The points in $\mathcal{L}'_{c,u} \backslash \mathcal{L}_{c,u}$ are independent of $u$. They are exactly the points $p \in V$ that satisfy $L_c(p)=L_{c,1}(p)=\ldots=L_{c,d-1}(p)=0$.
\end{lemma}

\begin{proof}
Any $p \in V$ satisfying $L_c(p)=L_{c,1}(p)=\ldots=L_{c,d-1}(p)=0$ is in $\mathcal{L}'_{c,u} \backslash \mathcal{L}_{c,u}$ for any $u$. Conversely, by the proof of~\cite[Theorem 13]{Carlos_etal}, if $p \in \mathcal{L}'_{c,u}\backslash \mathcal{L}_{c,u}$, then $L_c(p)=0$. It follows from the equations~(\ref{equations:linear_subspace}) that then also $L_{c,i}(p)=0$ for $i=1,\ldots,d-1$. But then $p$ satisfies equations~(\ref{equations:linear_subspace}) for any $u$.
\end{proof}

The more complicated reason for the ML degree drop can be the nontransversal intersection of $V$ and the linear subspace defined by~(\ref{equations:linear_subspace}). We recall that two projective varieties $A,B \subseteq \mathbb{P}^n$ intersect transversally at $p \in A \cap B$ if $p$ is a smooth point of $A,B$ and
$$
T_p A + T_p B = T_p \mathbb{P}^n.
$$
The intersection of $A$ and $B$ is generically transverse if it is transverse at a general point of every component of $A\cap B$. If the intersection of $V$ and the linear subspace defined by~(\ref{equations:linear_subspace}) is not generically transverse, the sum of degrees of the irreducible components of $\mathcal{L}'_{c,u}$ can be less than $\deg (V)$, in which case also the ML degree of the toric variety $V^c$ is less than the degree of the toric variety $V$.  One could think that the intersection of $V$ and the linear subspace defined by~(\ref{equations:linear_subspace}) is generically transverse for generic vectors $u$, but since the linear subspace defined by~(\ref{equations:linear_subspace}) depends on the variety $V$, then the intersection is not necessarily generically transverse. We will see several such examples later in this section and in Section~\ref{section:phylogenetic_models}. We note that the sum of degrees of the irreducible components can be less that $\deg V$ even if the degrees are counted with multiplicity as in~\cite[Example 12.3.1]{fulton1998intersection}.

\begin{corollary} \label{corollary:bound_on_ML_degree_drop}
The ML degree drop $\deg (V) - \mldeg (V^c)$ is bounded below by the sum of degrees of the irreducible components of the intersection of $V$ and the linear subspace defined by  $L_c(p)=L_{c,1}(p)=\ldots=L_{c,d-1}(p)=0$. If the intersection of $V$ and the linear subspace defined by~(\ref{equations:linear_subspace}) is generically transverse, then this bound is exact.
\end{corollary}

In Corollary~\ref{corollary:bound_on_ML_degree_drop}, we consider only irreducible components whose ideals are different from $\langle p_1,\ldots,p_n \rangle$ as we work over the projective space.

\begin{proof}
The sum of degrees of the irreducible components of $\mathcal{L}'_{c,u}$ is at most $\deg (V)$ by~\cite[Example 12.3.1]{fulton1998intersection} and the number of elements of $\mathcal{L}_{c,u}$ is $\mldeg (V^c)$. By Lemma~\ref{lemma:removed_points_independent}, we obtain $\mathcal{L}_{c,u}$ from $\mathcal{L}'_{c,u}$ by removing all the irreducible components that satisfy $L_c(p)=L_{c,1}(p)=\ldots=L_{c,d-1}(p)=0$. Hence the difference of the sum of degrees of the irreducible components of $\mathcal{L}'_{c,u}$ and the number of elements of $\mathcal{L}_{c,u}$ is the sum of degrees of the irreducible components of the intersection of $V$ and the linear subspace defined by  $L_c(p)=L_{c,1}(p)=\ldots=L_{c,d-1}(p)=0$. If $V$ and the linear subspace defined by~(\ref{equations:linear_subspace}) intersect generically transversely, then the sum of degrees of the irreducible components of $\mathcal{L}'_{c,u}$ is equal to $\deg V$.
\end{proof}

To understand the above observations, we analyze the different ML degree drops corresponding  to the quadratic Veronese embedding of $\PP^2$ given by the Fano polytope in Figure~\ref{polytope}.

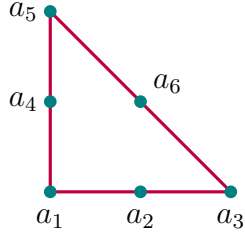
\begin{figure}[ht]

 \begin{center} \begin{tikzpicture}[scale=1.2]
\draw [purple,very thick](0,0) -- (0,2);
\draw [purple,very thick](0,0) -- (2,0);
\draw [purple,very thick](2,0) -- (0,2);

\node at (0,-.3) {$a_1$};
\node at (1,-.3) {$a_2$};
\node at (2,-.3) {$a_3$};
\node at (-.3,1) {$a_4$};
\node at (-.3,2) {$a_5$};
\node at (1.3,1.2) {$a_6$};

\fill[teal] (0,0) circle[radius=2pt];
\fill[teal] (0,1) circle[radius=2pt];
\fill[teal] (1,0) circle[radius=2pt];
\fill[teal] (1,1) circle[radius=2pt];
\fill[teal] (2,0) circle[radius=2pt];
\fill[teal] (0,2) circle[radius=2pt];
\end{tikzpicture}
\end{center}\caption{Polytope $Q$ corresponding to the smooth Fano variety $\PP^2$} \label{polytope}
\end{figure}

In \cite[Example 26]{Carlos_etal}, it was shown that different scalings $c \in \R^6$ produce ML degrees ranging from 1 to 4, under several combinations of the vector $c$ lying on each of the discriminants making up the principal $A$-determinant, defined by
\begin{align} \label{VerE_A}
E_A(c)&= \Delta_A(c) \cdot \Delta_{[a_1\; a_2\; a_3]}(c) \cdot \Delta_{[a_3\; a_5\; a_6]}(c) \cdot \Delta_{[a_1\; a_4\; a_5]}(c) \nonumber \\
&=\det (C) \det \begin{bmatrix}
2c_{00} & {c_{10}}  \\ {c_{10}} & 2c_{20}
\end{bmatrix}\det\begin{bmatrix}
2c_{20} & {c_{11}}  \\ {c_{11}} & 2c_{02}
\end{bmatrix}\det\begin{bmatrix}
2c_{00} & {c_{01}}  \\ {c_{01}} & 2c_{02}
\end{bmatrix},
\end{align}
where $C = \begin{bmatrix}
2c_{00} & {c_{10}} & {c_{01}} \\ {c_{10}} & 2c_{20} & {c_{11}} \\ {c_{01}} & {c_{11}} & 2c_{02}
\end{bmatrix}$.
The different combinations are presented in \cite[Table 2]{Carlos_etal}, which we reproduce here in Table \ref{table:Ver} (while fixing some typos). In each line, we go further than identifying a possible drop and actually explain the exact drops observed.

\begin{table}[ht]
\centering
\begin{tabular}{@{} l *8c @{}}
 \multicolumn{1}{c}{{ $C$}}    & { $\Delta_A$}  &   { $\Delta_{[a_1\; a_2\; a_3]}$ } &   { $\Delta_{[a_3\; a_5\; a_6]}$ } &   {  $\Delta_{[a_1\; a_4\; a_5]}$ } &   { ${\mldeg}$ }  \\
 \hline \\
   {$\begin{bmatrix} 2 & 1 & 1 \\ 1 & 2 & 1 \\ 1 & 1 & 2 \end{bmatrix}$}  & $\neq0$ & $\neq0$& $\neq0$& $\neq0$& \textbf{ 4}\vspace{2mm}\\
  $\begin{bmatrix} 2 & 2 & 1 \\ 2 & 2 & 3 \\ 1 & 3 & 2 \end{bmatrix}$  & $\neq0$ & 0& $\neq0$& $\neq0$& \textbf{ 3}\vspace{2mm}\\
     $\begin{bmatrix} 2 & 2 & 1 \\ 2 & 2 & 2 \\ 1 & 2 & 2 \end{bmatrix}$ & $\neq0$ & 0 & 0& $\neq0$& \textbf{ 2}\vspace{2mm}\\
   {\small$\begin{bmatrix} \text{-}2 & 2 & 2 \\ 2 & \text{-}2 & 2 \\ 2 & 2 & \text{-}2 \end{bmatrix}$}  & $\neq0$ & 0& 0& 0 & \textbf{ 1}\vspace{2mm}\\
   {\small$\begin{bmatrix} 17 & 22 & 27 \\ 22 & 29 & 36 \\ 27 & 36 & 45 \end{bmatrix}$} & 0 & $\neq0$ & $\neq0$ & $\neq0$& \textbf{ 3}\vspace{2mm}\\
  $\begin{bmatrix} 2 & 3 & 3 \\ 3 & 5 & 5 \\ 3 & 5 & 5 \end{bmatrix}$ & 0 & $\neq0$ & 0& $\neq0$& \textbf{ 2}\vspace{2mm}\\
  $\begin{bmatrix} 2 & 2 & 2 \\ 2 & 2 & 2 \\ 2 & 2 & 2 \end{bmatrix}$ & 0 & 0 & 0& 0 & \textbf{ 1}\vspace{2mm}\\
 \end{tabular}\vspace{1mm}
\caption{ML degrees for different scalings $c_{ij}$ in the matrix $C$.} \label{table:Ver}
 \end{table}

Naively, each appearance of a $0$ in a row of Table~\ref{table:Ver} makes the ML degree drop by $1$. But this cannot be, since the last row has all four zeros and the ML degree cannot drop to $0$. We will see in the explanation of the last two rows that it is in general impossible to predict the exact drop  just from knowing in what discriminants the vector $c$ lies.

\begin{itemize}

\item \textbf{Row 1} This corresponds to the generic case. The intersection $\mathcal{L}'_{c,u}$ is transverse and zero-dimensional with 4 points, corresponding to the ML degree. There are no points in $\mathcal{L}'_{c,u} \backslash \mathcal{L}_{c,u}$ and there is no drop.

\item \textbf{Row 2} When computing the points in $V_A \cap \{ L_{c}(p)=L_{c,1}(p)=\ldots=L_{c,d-1}(p)=0 \}$ we obtain the unique projective point $[1:-1:1:0:0:0]$, which makes the $A$-discriminant of the edge $[a_1,a_2,a_3]$ of $Q$ vanish. Removing this point gives the ML degree of $3 = 4 - 1$.

\item \textbf{Row 3} Now we have one more point in the removal set: apart from the one in the above row, there is also $[0:0:1:0:1:-1]$ on the zero locus of the $A$-discriminant of the edge $[a_3,a_5,a_6]$. The drop is accounted for exactly these two points and we have ML degree $2 = 4 - 2$.

\item \textbf{Row 4} There are three points in $\mathcal{L}'_{c,u} \backslash \mathcal{L}_{c,u}$ that lie on the zero loci of edge $A$-discriminants: $[1:1:1:0:0:0]$ for the edge $[a_1,a_2,a_3]$,  $[0:0:1:0:1:1]$ for the edge $[a_3,a_5,a_6]$ and $[1:0:0:1:1:0]$ for the edge $[a_1,a_4,a_5]$. They explain the drop in ML degree $1 = 4 - 3 $.

\item \textbf{Row 5} The only removal point is $[1:-2:4:1:1:-2]$, which does \textit{not} lie on zero loci of any of the $A$-discriminants of the edges of $Q$, but only on the zero locus of the $A$-discriminant of the whole of $Q$. Removing this point gives the ML degree of $3 = 4 - 1$.

\item \textbf{Row 6} This is the first time that the removal ideal $I_A + \left\langle L_{c}(p),L_{c,1}(p),\ldots,L_{c,d-1}(p) \right\rangle$ is not radical. While there is only one point, $[0:0:1:0:1:-1]$, its multiplicity is 2. We used the \texttt{Macaulay2} package \texttt{SegreClasses}~\cite{harris2020segre} to compute the multiplicity. The intersection $\mathcal{L}'_{c,u}$  is zero-dimensional but consists of two components of degree 2. The first component is prime and corresponds to the two points in the ML degree $2$. The toric variety and the linear space defined by~\eqref{equations:linear_subspace} intersect transversely at both points of the first component. The second component is primary, but not prime. Its radical $\langle p_6 + p_5, p_4, p_3 + p_5, p_2, p_1 \rangle$ is a zero-dimensional ideal of degree $1$, corresponding to the above point that lies on the zero locus of the $A$-discriminant of the edge $[a_3,a_5,a_6]$.

Although the intersection of the toric variety and the linear space defined by~\eqref{equations:linear_subspace} is dimensionally transverse, it is not transverse at the point defined by the second component. We also observe that while $\Delta_A(c)=0$, there is no singular point of $f_c$, which means $c$ lies strictly in the closure in Definition~\ref{def:Adisc} (see Remark \ref{rmk:row6} below).

\item \textbf{Row 7} Now the removal ideal is 1-dimensional of degree 2. It is given by $$\left\langle p_2^2 + p_2p_3 + p_3p_4 \, , \, p_1+p_2+p_4 \, , \, p_2+p_5-p_2-p_3 \, , \, p_2 + p_3 + p_6 \right\rangle.$$ Its variety intersects $V_{\Gamma \cap A}$ in one point for each edge $\Gamma$ of $Q$. In other words, the reason why all discriminants $\Delta_{[a_1\; a_2\; a_3]},  \Delta_{[a_3\; a_5\; a_6]}, \Delta_{[a_1\; a_4\; a_5]}$ vanish is that the removal set intersects the planes $p_1=p_2=p_4=0 \, , \, p_2=p_3=p_6=0 $ and $p_4=p_5=p_6=0$ respectively, and one can find in each a point with complementary support. Furthermore, it intersects the open set where none of the $p_i$ are zero, which explains why $\Delta_A=0$ too. Unfortunately, this alone does not explain why the ML degree is 1.

By looking at the intersection ideal of the toric variety $V_A$ with the equations \eqref{equations:linear_subspace}, we realize that the intersection is not transverse (not even dimensionally transverse). Indeed, there are two components: a zero-dimensional component of degree 1 (corresponding to the MLE) and a one-dimensional component of degree 2 (so the sum of the degrees is $1+2=3<4$). This last component matches the removal ideal above. At the 0-dimensional component the toric variety intersects the linear space defined by~\eqref{equations:linear_subspace} transversely. At a generic point of the $1$-dimensional component the intersection is not transverse. Both components have multiplicity one and hence also the sum of degrees counted with multiplicity is less than four.
\end{itemize}

\begin{remark}\label{rmk:row6}
If for some scaling $c$, the $A$-discriminant of at least one edge is zero and the $A$-discriminant of at least one edge is non-zero, then there is no singular point $\theta \in \left(\mathbb{C}^{*}\right)^2$ of $f_c$. Indeed, such a point $\theta=(\theta_1,\theta_2) \in \left(\mathbb{C}^{*}\right)^2$ would need to satisfy
\begin{equation}\label{system} \begin{bmatrix}
2c_{00} & {c_{10}} & {c_{01}} \\ {c_{10}} & 2c_{20} & {c_{11}} \\ {c_{01}} & {c_{11}} & 2c_{02}
\end{bmatrix} \begin{bmatrix}
1 \\ \theta_1 \\ \theta_2
\end{bmatrix} = \begin{bmatrix}
0 \\ 0 \\ 0
\end{bmatrix}.
\end{equation}
Say $\Delta_{[a_3\; a_5\; a_6]} =$ \begin{small}$\det\begin{bmatrix}
2c_{20} & {c_{11}}  \\ {c_{11}} & 2c_{02}
\end{bmatrix}=0$ \end{small}. In order for the system \eqref{system} to be consistent, the rank of \begin{small}$\begin{bmatrix}
 {c_{10}} & 2c_{20} & {c_{11}} \\ {c_{01}} & {c_{11}} & 2c_{02}
\end{bmatrix}$\end{small} should be $1$. In particular, $\det\begin{bmatrix}
{c_{10}} & {c_{01}}  \\ {2c_{20}} & c_{11}
\end{bmatrix}=0$ and  $\det\begin{bmatrix}
{c_{10}} & {c_{01}}  \\ {c_{11}} & 2c_{02}
\end{bmatrix}=0$, which in turn force \begin{small}$\begin{bmatrix}
2c_{00} & {c_{10}} & {c_{01}} \\ {c_{10}} & 2c_{20} & {c_{11}}
\end{bmatrix}$\end{small} and  \begin{small}$\begin{bmatrix}
2c_{00} & {c_{10}} & {c_{01}} \\ {c_{01}} & {c_{11}} & {2c_{02}}
\end{bmatrix}$\end{small} to have rank $1$. We conclude that $C$ itself must have rank $1$, so that $c$ must lie in the discriminants of all the edges (as in Row 7), contradicting that one of them was non-zero. This means that if we are in the case of Row 6, even though $\Delta_A(c)=0$, the scaling $c \in \nabla_A$ is added in the closure that appears in Definition \ref{def:Adisc}.
\end{remark}

\begin{conjecture}
The intersection of $V$ and the linear space given by~\eqref{equations:linear_subspace} is generically transverse at the irreducible component of $\mathcal{L}'_{c,u}$ that gives the MLE.
\end{conjecture}

This conjecture holds for all the examples considered in this section. At other irreducible components the intersection may or may not be transverse.

\begin{remark}
For toric models,
\begin{equation}\label{eq:mltop}
\mldeg(V) = \chi_{\text{top}}(V \backslash H) = \chi_{\text{top}}(V) - \chi_{\text{top}}(V \cap H)
\end{equation}
 where $H = \{p\, | \, p_1  \cdots  p_n (p_1 + \dots + p_n)= 0 \}$
and $\chi_{\text{top}}$ is the topological Euler characteristic. The first equality was proved by Huh and Sturmfels \cite[Theorem 3.2]{Huh_Sturmfels}, and the second equality is the excision formula. The Euler characteristic of toric del Pezzo surfaces can be computed by $\chi_{\text{top}}(V) = 2 + \text{rkPic}(V)$. The rank of the Picard group $\text{rkPic}(V)$ can be computed taking into account that the minimal resolution of every del Pezzo surface is a product of two projective lines $\mathbb{P}^1 \times \mathbb{P}^1$ (polytope 8a in Table~\ref{table:polytopes}), the quadric cone $\mathbb{P}(1,1,2) \subset \mathbb{P}^3$ (polytope 8c in Table~\ref{table:polytopes}), or the blow-up of a projective plane in $9-d$ points in almost general position; namely at most three of which are collinear, at most six of which lie on a conic, and at most eight of them on a cubic having a node at one of the points. We refer the reader to \cite{Dolgachev} for a more detailed study of this classical subject of algebraic geometry. It would be interesting to explore the ML degree drop further from this perspective.
\end{remark}

\section{Toric fiber products and phylogenetic models} \label{section:phylogenetic_models}

In~\cite{Wisniewski_Buczynska_2007}, Buczy\'{n}ska-Wi\'{s}niewski study an infinite family of toric Fano varieties that correspond to 3-valent phylogenetic trees. These Fano varieties are of index $4$ with Gorenstein terminal singularities. In phylogenetics, these varieties correspond to the CFN model in the Fourier coordinates. We define them through the corresponding polytopes.

\begin{definition} \label{def:phylogenetic_models}
Let $\mathcal{T}$ be a 3-valent tree, i.e. every vertex of $\mathcal{T}$ has degree $1$ or $3$. Consider all labelings of the edges of $\mathcal{T}$ with $0$'s and $1$'s such that at every inner vertex the sum of labels on the incident edges is even. Define $P_{\mathcal{T}} \subseteq \R^{E}$ to be the convex hull of such labelings. Let $I_{\mathcal{T}}$ be the homogeneous ideal and $V_{\mathcal{T}}$ be the projective toric variety corresponding to $P_{\mathcal{T}}$.
\end{definition}

\begin{example}
If $\mathcal{T}$ is tripod, then $P_{\mathcal{T}}=\conv((0,0,0),(1,1,0),(1,0,1),(0,1,1))$ is a simplex and $V_{\mathcal{T}}=\mathbb{P}^3$ is the $3$-dimensional projective space.
\end{example}

\begin{example}
If $\mathcal{T}$ is the unique 3-valent four leaf tree, then
\begin{align*}
P_{\mathcal{T}}=\conv(&(0,0,0,0,0),(1,1,0,0,0),(0,0,0,1,1),(1,1,0,1,1),\\
&(1,0,1,1,0),(1,0,1,0,1),(0,1,1,1,0),(0,1,1,0,1)).
\end{align*}
The ideal $I_{\mathcal{T}}$ is generated by the two quadratic polynomials $x_{00000}x_{11011}-x_{11000}x_{00011}$ and $x_{10110}x_{01101}-x_{01110}x_{10101}$.
\end{example}

The aim of the rest of the section is to show that if $\mathcal{T}$ is a 3-valent tree then the ML degree of the variety $V_{\mathcal{T}}$ is one. We will also give a closed form for its maximum likelihood estimate. This result will be a special case of a more general result about ML degrees
 of codimension-0 toric fiber products of toric ideals. A toric fiber product can be defined for any two ideals that are homogenous by the same multigrading~\cite{sullivant2007toric}, however, we use the definition specific to toric ideals~\cite[Section 2.3]{engstrom2014multigraded}. Besides phylogenetic models considered in this section, codimension-0 toric fiber products appear in general group-based models in the Fourier coordinates and reducible hierarchical models (see~\cite[Section 3]{sullivant2007toric} for details on applications).

Let $r \in \N$ and $s_i, t_i \in \N$ for $1 \leq i \leq r$. We consider toric ideals corresponding to vector configurations $\mathcal{B}=\{b^i_j:i \in [r], j \in [s_i]\} \subseteq \Z^{d_1}$ and $\mathcal{C}=\{c^i_k:i \in [r], k \in [t_i]\} \subseteq \Z^{d_2}$. These toric ideals are denoted by $I_{\mathcal{B}}$ and $I_{\mathcal{C}}$, they live in the polynomial rings $\R[x^i_j:i \in [r], j \in [s_i]]$ and $\R[y^i_k:i \in [r], k \in [t_i]]$, and they are required to be homogeneous with respect to the multigrading by $\mathcal{A}=\{a^i:i \in [r]\} \subseteq \Z^d$. We assume that there exists $\omega \in \Q^d$ such that $\omega a^i=1$ for all $i$, so that $I_{\mathcal{B}}$ and $I_{\mathcal{C}}$ are homogeneous also with respect to the standard grading. Toric ideals $I_{\mathcal{B}}$ and $I_{\mathcal{C}}$ being homogeneous with respect to the multigrading by $\mathcal{A}$ implies that there exist linear maps $\pi_1: \Z^{d_1} \rightarrow \Z^d$ and $\pi_2: \Z^{d_2} \rightarrow \Z^d$ such that $\pi_1(b^i_j)=a^i$ and $\pi_2(c^i_k)=a^i$. We define the vector configuration
$$
\mathcal{B} \times_{\mathcal{A}} \mathcal{C} = \{(b^i_j,c^i_k):i\in [r],j \in [s_i], k \in [t_i]\}.
$$
The toric fiber product of $I_{\mathcal{B}}$ and $I_{\mathcal{C}}$ with respect to the multigrading by $\mathcal{A}$ is defined as
$$
I_{\mathcal{B}} \times_{\mathcal{A}} I_{\mathcal{C}} := I_{\mathcal{B} \times_{\mathcal{A}} \mathcal{C}},
$$
and it lives in the polynomial ring $\R[z^i_{jk}:i \in [r], j \in [s_i], k \in [t_i]]$.

\begin{example}
Let $\mathcal{T}$ be a 3-valent tree with $n \geq 4$ leaves. Write $\mathcal{T}$ as a union of two trees $\mathcal{T}_1$ and $\mathcal{T}_2$ that share an interior edge $e$. Take $r=2$. Let $b^1_j$ be the vertices of $P_{{\mathcal{T}}_1}$ that have label $0$ on edge $e$ and let $b^2_j$ be the vertices of $P_{{\mathcal{T}}_1}$ that have label $1$ on edge $e$. Define similarly $c^1_k$ and $c^2_k$ for $P_{{\mathcal{T}}_2}$. Let $\mathcal{A}=\{(0,1),(1,0)\}$, so that $\pi_1$ maps $b^1_j$ to $(0,1)$ and $b^2_j$ to $(1,0)$. Then $I_{\mathcal{T}}$ is the toric fiber product of $I_{{\mathcal{T}}_1}$ and $I_{{\mathcal{T}}_2}$ with respect to multigrading by $\mathcal{A}$. In~\cite[Section 3.4]{sullivant2007toric} the toric fiber product construction is explained in full generality for group-based phylogenetic models in the Fourier coordinates.
\end{example}

Given a vector $u$ indexed by the elements of $\mathcal{B} \times_{\mathcal{A}} \mathcal{C}$, we denote its entries by $u^i_{jk}$ for $i\in [r],j \in [s_i], k \in [t_i]$. We define $u^i_{++} = \sum_{j \in [s_i], k \in [t_i]} u^i_{jk}$,  $u^i_{j+} = \sum_{k \in [t_i]} u^i_{jk}$ and $u^i_{+k} = \sum_{j \in [s_i]} u^i_{jk}$. We denote the corresponding vectors by $u_{\mathcal{A}},u_{\mathcal{B}}$ and $u_{\mathcal{C}}$ since they are indexed by the elements of $\mathcal{A}, \mathcal{B}$ and $\mathcal{C}$. These vectors $u_{\mathcal{A}} = (u^i_{++})_{i\in [r]}$, $u_{\mathcal{B}} = (u^i_{j+})_{i\in [r],j \in [s_i]}$ and $u_{\mathcal{C}}=(u^i_{+k})_{i\in [r], k \in [t_i]}$ are marginal sums. We also define  $u^+_{++} = \sum_{i \in [r], j \in [s_i], k \in [t_i]} u^i_{jk}$, $(u_{\mathcal{B}})_{+}^{+} = \sum_{i\in [r]} \sum_{j \in [s_i]} (u^i_{j+})$ and  $(u_{\mathcal{C}})_+^+ = \sum_{i\in [r]} \sum_{k \in [t_i]} (u^i_{+k})$.  Similarly, if $p^i_{jk}$ is a joint probability distribution indexed by the elements of $\mathcal{B} \times_{\mathcal{A}} \mathcal{C}$, the sum of the joint probability distribution over $\mathcal{A}$ (resp. $\mathcal{B}$, $\mathcal{C}$) is a marginal probability distribution and we denote it by $p_{\mathcal{A}} = (p^i_{++})_{i\in [r]}$ (resp. $p_{\mathcal{B}} = (p^i_{j+})_{i\in [r],j \in [s_i]}$, $p_{\mathcal{C}}=(p^i_{+k})_{i\in [r], k \in [t_i]}$).

\begin{theorem} \label{theorem:ML_degree_of_toric_fiber_products}
Let $\mathcal{A}$ consist of linearly independent vectors. Then the ML degree of $I_{\mathcal{B}} \times_{\mathcal{A}} I_{\mathcal{C}}$ is equal to the product of the ML degrees of $I_{\mathcal{B}}$ and $I_{\mathcal{C}}$. For a data vector $u$, every critical point of the likelihood function has the form
\begin{equation} \label{formula:critical_point_of_the_likelihood_function}
\overline{p}^i_{jk} := \frac{(\overline{p_{\mathcal{B}}})^i_{j} (\overline{p_{\mathcal{C}\,}})^i_{k}}{(\overline{p_{\mathcal{A}}})^i},
\end{equation}
where $\overline{p_{\mathcal{A}}}$,  $\overline{p_{\mathcal{B}}}$ and $\overline{p_{\mathcal{C}\,}}$ are  critical points of the likelihood function for the models $I_{\mathcal{A}}$, $I_{\mathcal{B}}$ and $I_{\mathcal{C}}$ and data vectors $u_{\mathcal{A}}$, $u_{\mathcal{B}}$ and $u_{\mathcal{C}}$, respectively. Since the elements of $\mathcal{A}$ are linearly independent, $\overline{p_{\mathcal{A}}}$ is in fact the normalized $u_{\mathcal{A}}$. Moreover, we obtain the maximum likelihood estimate of $I_{\mathcal{B}} \times_{\mathcal{A}} I_{\mathcal{C}}$ by taking $\overline{p_{\mathcal{A}}}$,  $\overline{p_{\mathcal{B}}}$ and $\overline{p_{\mathcal{C}\,}}$ to be the maximum likelihood estimates of the models $I_{\mathcal{A}}$, $I_{\mathcal{B}}$ and $I_{\mathcal{C}}$.
\end{theorem}

Theorem~\ref{theorem:ML_degree_of_toric_fiber_products} generalizes known results about reducible hierarchical~\cite[Proposition 4.14]{lauritzen1996graphical} and discrete graphical models. In particular, one recovers the rational formula for the MLE in the case of decomposable graphical models (indeed, they have ML degree one)~\cite[ Proposition 4.18]{lauritzen1996graphical} and general discrete graphical models~\cite[Theorem 1]{frydenberg1989decomposition}. The latter result is for mixed graphical interaction models that allow both discrete and continuous random variables. Theorem~\ref{theorem:ML_degree_of_toric_fiber_products} generalizes the case when all variables are discrete.

To prove Theorem~\ref{theorem:ML_degree_of_toric_fiber_products}, we first have to recall how to obtain a generating set for $I_{\mathcal{B}} \times_{\mathcal{A}} I_{\mathcal{C}}$ from the generating sets for $I_{\mathcal{B}}$ and $I_{\mathcal{C}}$. Let
$$
f = x^{i_1}_{j^1_1} x^{i_2}_{j^1_2} \cdots x^{i_d}_{j^1_d} - x^{i_1}_{j^2_1} x^{i_2}_{j^2_2} \cdots x^{i_d}_{j^2_d} \in \R[x^i_j].
$$
For $k=(k_1,k_2,\ldots,k_d) \in [t_{i_1}] \times [t_{i_2}] \times \cdots \times [t_{i_d}]$ define
$$
f_k = z^{i_1}_{j^1_1 k_1} z^{i_2}_{j^1_2 k_2} \cdots z^{i_d}_{j^1_d k_d} - z^{i_1}_{j^2_1 k_1} z^{i_2}_{j^2_2 k_2} \cdots z^{i_d}_{j^2_d k_d} \in \R[z^i_{jk}].
$$
Let $T_f = \prod_{l=1}^d [t_{i_l}]$.  For a set of binomials $F$, we define
$$
\text{Lift}(F) = \{f_k: f \in F, k \in T_f\}.
$$
We also define
$$
\text{Quad}=\{z^i_{j_1 k_1} z^i_{j_2 k_2} - z^i_{j_1 k_2} z^i_{j_2 k_1}: i \in [r], j_1,j_2 \in [s_i], k_1,k_2 \in [t_i]\}.
$$

\begin{proposition}[\cite{sullivant2007toric}, Corollary 14] \label{theorem:toric_fiber_product_generators}
Let $\mathcal{A}$ consist of linearly independent vectors. Let $F$ be a generating set of $I_{\mathcal{B}}$ and let $G$ be a generating set of $I_{\mathcal{C}}$. Then $I_{\mathcal{B}} \times_{\mathcal{A}} I_{\mathcal{C}}$ is generated by
$$
\text{Lift}(F) \cup \text{Lift}(G) \cup \text{Quad}.
$$
\end{proposition}

\begin{example}
The 3-valent four leaf tree $\mathcal{T}_4$ is the union of two tripods $\mathcal{T}_3$ that share an interior edge $e$. By Proposition~\ref{theorem:toric_fiber_product_generators}, a generating set for $I_{{\mathcal{T}}_4}$ is given by the lifts of generating sets for $I_{{\mathcal{T}}_3}$ and by quadrics with respect to the edge $e$. Since  $I_{{\mathcal{T}}_3}=\langle 0 \rangle$, its lift is $\{0\}$. The set $\text{Quad}$ consists of $x_{00000}x_{11011}-x_{11000}x_{00011}$ and $x_{10110}x_{01101}-x_{10101}x_{01110}$ that are generators of $I_{{\mathcal{T}}_4}$.
\end{example}

\begin{example}
The 3-valent five leaf tree $\mathcal{T}_5$ is the union of the 3-valent four leaf tree $\mathcal{T}_4$ and tripod $\mathcal{T}_3$ that share an interior edge $e$. The fifth index of a variable $x_{e_1 e_2 e_3 e_4 e_5}$ in the coordinate ring of $\mathcal{T}_4$ and the first index of a variable $x_{e_5 e_6 e_7}$ in the coordinate ring of $\mathcal{T}_3$ correspond to the edge $e$. Recall that a generating set of $I_{{\mathcal{T}}_4}$ is $F=\{x_{00000}x_{11011}-x_{11000}x_{00011},x_{10110}x_{01101}-x_{01110}x_{10101}\}$ and a generating set of $I_{{\mathcal{T}}_3}$ is $G=\{0\}$. Both elements of $F$ have four lifts corresponding to $k=(000,110),k=(000,101),k=(011,110)$ and $k=(011,101)$. Hence $\text{Lift}(F)$ consists of
\begin{align*}
x_{0000000}x_{1101110}-x_{1100000}x_{0001110}, x_{0000000}x_{1101101}-x_{1100000}x_{0001101},\\
x_{0000011}x_{1101110}-x_{1100011}x_{0001110}, x_{0000011}x_{1101101}-x_{1100011}x_{0001101},\\
x_{1011000}x_{0110110}-x_{0111000}x_{1010110},x_{1011000}x_{0110101}-x_{0111000}x_{1010101},\\
x_{1011011}x_{0110110}-x_{0111011}x_{1010110},x_{1011011}x_{0110101}-x_{0111011}x_{1010101},
\end{align*}
and $\text{Lift}(G)=\{0\}$. The set $\text{Quad}$ consists of 12 polynomials.
\end{example}

To prove Theorem~\ref{theorem:ML_degree_of_toric_fiber_products}, we also need the following lemmas.

\begin{lemma} \label{lemma:multiplication_with_the_fiber_product_matrix}
For any $u \in \mathbb{R}^{\mathcal{B} \times_{\mathcal{A}} \mathcal{C}}$, we have $(\mathcal{B} \times_{\mathcal{A}} \mathcal{C}) u = (\mathcal{B} u_{\mathcal{B}},\mathcal{C} u_{\mathcal{C}})$.
\end{lemma}

\begin{proof}
Assume that the $r$-th row comes from the $\mathcal{B}$ part of the matrix $\mathcal{B} \times_{\mathcal{A}} \mathcal{C}$. Then the $r$-th row of $\mathcal{B} \times_{\mathcal{A}} \mathcal{C}$ multiplied with $u$ gives
$$
\sum_{i \in [r], j \in [s_i], k \in [t_i]} (b^i_j,c^i_k)_r u^i_{jk} = \sum_{i \in [r], j \in [s_i], k \in [t_i]} (b^i_j)_r u^i_{jk} = \sum_{i \in [r], j \in [s_i]} (b^i_j)_r u^i_{j+}.
$$
This is the $r$-th row of $\mathcal{B}$ multiplied with $u_{\mathcal{B}}$.
\end{proof}

\begin{lemma}
\label{lemma:marginaL_{c,i}s_MLE}
The following equations hold:
$$
\overline{p}_{\mathcal{B}} = \overline{p_{\mathcal{B}}} \quad \text{and} \quad \overline{p}_{\mathcal{C}} = \overline{p_{\mathcal{C}\,}}.
$$
In particular, the entries of $\overline{p}$ sum to one i.e. $\overline{p}_{++}^{+}= \sum_{i\in [r]} \sum_{j \in [s_i]}  \sum_{k \in [t_i]} \overline{p}_{jk}^{i} = 1$ .
\end{lemma}

\begin{proof}
By the definition of $\overline{p}$, we have
$$
(\overline{p}_{\mathcal{B}})^i_{j} =  \sum_{k \in [t_i]}\frac{(\overline{p_{\mathcal{B}}})^i_{j} (\overline{p_{\mathcal{C}\,}})^i_{k}}{(\overline{p_{\mathcal{A}}})^i} = \frac{(\overline{p_{\mathcal{B}}})^i_{j} (\overline{p_{\mathcal{C}\,}})^i_{+}}{(\overline{p_{\mathcal{A}}})^i} \text{ .}
$$
Hence we need to show that $(\overline{p_{\mathcal{C}\,}})^i_{+}=(\overline{p_{\mathcal{A}}})^i$. By Birch's theorem for $I_{\mathcal{C}}$, we have $\mathcal{C} \overline{p_{\mathcal{C}\,}}=\mathcal{C} \frac{u_{\mathcal{C}}}{{(u_{\mathcal{C}})^+_+}}$ and hence $\pi_2(\mathcal{C} )\overline{p_{\mathcal{C}\,}}=\pi_2(\mathcal{C})\frac{u_{\mathcal{C}}}{{(u_{\mathcal{C}})^+_+}}$ where $\pi_2$ is applied to $C$ columnwise. The second equation is equivalent to $\sum_{i \in [r]} t_i a^i (\overline{p_{\mathcal{C}\,}})^i_+ = \sum_{i \in [r]} t_i a^i  \frac{(u_{\mathcal{C}})^i_{+}}{{(u_{\mathcal{C}})^+_+}}$. Since $a^i$ are linearly independent, this implies $(\overline{p_{\mathcal{C}\,}})^i_+ =\frac{(u_{\mathcal{C}})^i_{+}}{{(u_{\mathcal{C}})^+_+}}=\frac{u^i_{++}}{{u^+_{++}}}=(\overline{p_{\mathcal{A}}})^i$ for all $i \in [r]$.
\end{proof}

\begin{proof}[Proof of Theorem~\ref{theorem:ML_degree_of_toric_fiber_products}]
We start by showing that every vector of the form (\ref{formula:critical_point_of_the_likelihood_function}) satisfies the conditions of Birch's theorem, i.e. $\overline{p}^+_{++}=1$, $(\mathcal{B} \times_{\mathcal{A}} \mathcal{C}) \overline{p} = (\mathcal{B} \times_{\mathcal{A}} \mathcal{C}) \frac{u}{u^+_{++}}$ and $\overline{p} \in V(I_{\mathcal{B}} \times_{\mathcal{A}} I_{\mathcal{C}})$, and hence is a critical point of the likelihood function of $I_{\mathcal{B}} \times_{\mathcal{A}} I_{\mathcal{C}}$. The entries of $\overline{p}$ sum to one by Lemma~\ref{lemma:marginaL_{c,i}s_MLE}. Secondly, we have
$$
(\mathcal{B} \times_{\mathcal{A}} \mathcal{C}) \overline{p} = (\mathcal{B} \overline{p}_{\mathcal{B}},\mathcal{C} \overline{p}_{\mathcal{C}}) = (\mathcal{B} \overline{p_{\mathcal{B}}},\mathcal{C} \overline{p_{\mathcal{C}\,}}) = (\mathcal{B} \frac{u_{\mathcal{B}}}{(u_{\mathcal{B}})^+_+},\mathcal{C} \frac{u_{\mathcal{C}}}{(u_{\mathcal{C}})^+_+}) = (\mathcal{B} \frac{u_{\mathcal{B}}}{u^+_{++}},\mathcal{C} \frac{u_{\mathcal{C}}}{u^+_{++}}) = (\mathcal{B} \times_{\mathcal{A}} \mathcal{C}) \frac{u}{u^+_{++}}.
$$
The first and last equalities hold by Lemma~\ref{lemma:multiplication_with_the_fiber_product_matrix}. The second equality holds by Lemma~\ref{lemma:marginaL_{c,i}s_MLE} while the third equality follows from Birch's theorem for $I_{\mathcal{B}}$ and $I_{\mathcal{C}}$.
Thirdly, we have to show $\overline{u} \in  I_{\mathcal{B}} \times_{\mathcal{A}} I_{\mathcal{C}}$. For $f_k \in \text{Lift}(F)$, we have
\begin{align*}
f_k(\overline{p}) &= \frac{(\overline{p_{\mathcal{B}}})^{i_1}_{j^1_1} (\overline{p_{\mathcal{C}\,}})^{i_1}_{k_1}}{(\overline{p_{\mathcal{A}}})^{i_1}} \frac{(\overline{p_{\mathcal{B}}})^{i_2}_{j^1_2} (\overline{p_{\mathcal{C}\,}})^{i_2}_{k_2}}{(\overline{p_{\mathcal{A}}})^{i_2}} \cdots \frac{(\overline{p_{\mathcal{B}}})^{i_d}_{j^1_d} (\overline{p_{\mathcal{C}\,}})^{i_d}_{k_d}}{(\overline{p_{\mathcal{A}}})^{i_d}} - \frac{(\overline{p_{\mathcal{B}}})^{i_1}_{j^2_1} (\overline{p_{\mathcal{C}\,}})^{i_1}_{k_1}}{(\overline{p_{\mathcal{A}}})^{i_1}} \frac{(\overline{p_{\mathcal{B}}})^{i_2}_{j^2_2} (\overline{p_{\mathcal{C}\,}})^{i_2}_{k_2}}{(\overline{p_{\mathcal{A}}})^{i_2}} \cdots \frac{(\overline{p_{\mathcal{B}}})^{i_d}_{j^2_d} (\overline{p_{\mathcal{C}\,}})^{i_d}_{k_d}}{(\overline{p_{\mathcal{A}}})^{i_d}}\\
&= \left((\overline{p_{\mathcal{B}}})^{i_1}_{j^1_1}  (\overline{p_{\mathcal{B}}})^{i_2}_{j^1_2}  \cdots (\overline{p_{\mathcal{B}}})^{i_d}_{j^1_d}  - (\overline{p_{\mathcal{B}}})^{i_1}_{j^2_1} (\overline{p_{\mathcal{B}}})^{i_2}_{j^2_2}  \cdots (\overline{p_{\mathcal{B}}})^{i_d}_{j^2_d} \right) \frac{ (\overline{p_{\mathcal{C}\,}})^{i_1}_{k_1}}{(\overline{p_{\mathcal{A}}})^{i_1}} \frac{(\overline{p_{\mathcal{C}\,}})^{i_2}_{k_2}}{(\overline{p_{\mathcal{A}}})^{i_2}} \cdots \frac{(\overline{p_{\mathcal{C}\,}})^{i_d}_{k_d}}{(\overline{p_{\mathcal{A}}})^{i_d}}\\
&=f(\overline{p_{\mathcal{B}}}) \cdot \frac{ (\overline{p_{\mathcal{C}\,}})^{i_1}_{k_1}}{(\overline{p_{\mathcal{A}}})^{i_1}} \frac{(\overline{p_{\mathcal{C}\,}})^{i_2}_{k_2}}{(\overline{p_{\mathcal{A}}})^{i_2}} \cdots \frac{(\overline{p_{\mathcal{C}\,}})^{i_d}_{k_d}}{(\overline{p_{\mathcal{A}}})^{i_d}} = 0.
\end{align*}
An element of $\text{Quad}$ gives
\begin{align*}
(z^i_{j_1 k_1} z^i_{j_2 k_2} - z^i_{j_1 k_2} z^i_{j_2 k_1})(\overline{p}) =
\frac{(\overline{p_{\mathcal{B}}})^i_{j_1} (\overline{p_{\mathcal{C}\,}})^i_{k_1}}{(\overline{p_{\mathcal{A}}})^{i}} \frac{(\overline{p_{\mathcal{B}}})^i_{j_2} (\overline{p_{\mathcal{C}\,}})^i_{k_2}}{(\overline{p_{\mathcal{A}}})^{i}} - \frac{(\overline{p_{\mathcal{B}}})^i_{j_1} (\overline{p_{\mathcal{C}\,}})^i_{k_2}}{(\overline{p_{\mathcal{A}}})^{i}} \frac{(\overline{p_{\mathcal{B}}})^i_{j_2} (\overline{p_{\mathcal{C}\,}})^i_{k_1}}{(\overline{p_{\mathcal{A}}})^{i}} =0.
\end{align*}
Hence $\overline{p}$ is a critical point of the likelihood function of $I_{\mathcal{B}} \times_{\mathcal{A}} I_{\mathcal{C}}$.

Conversely, let $\overline{p}$ be any critical point of the likelihood function of $I_{\mathcal{B}} \times_{\mathcal{A}} I_{\mathcal{C}}$. Then the  entries of $\overline{p}_{\mathcal{B}}$ and $\overline{p}_{\mathcal{C}}$ sum to one. Also
$$
(\mathcal{B} \frac{u_{\mathcal{B}}}{(u_{\mathcal{B}})^+_+},\mathcal{C} \frac{u_{\mathcal{C}}}{(u_{\mathcal{C}})^+_+}) = (\mathcal{B} \frac{u_{\mathcal{B}}}{u^+_{++}},\mathcal{C} \frac{u_{\mathcal{C}}}{u^+_{++}}) = (\mathcal{B} \times_{\mathcal{A}} \mathcal{C}) \frac{u}{u^+_{++}} = (\mathcal{B} \times_{\mathcal{A}} \mathcal{C}) \overline{p} =(\mathcal{B} \overline{p}_{\mathcal{B}},\mathcal{C} \overline{p}_{\mathcal{C}}).
$$
For every $f$ in a generating set for $I_{\mathcal{B}}$
$$
f(\overline{p}_{\mathcal{B}})= \sum_{k \in T_f} c_k f_k  (\overline{p}) =0,
$$
where $c_k$ are integer coefficients. By Birch's theorem, $\overline{p}_B$ is a critical point of the likelihood function of $I_{\mathcal{B}}$ and similarly $\overline{p}_C$ is a critical point of the likelihood function of $I_{\mathcal{C}}$. It is left to show that $\overline{p}$ is the only element in $ I_{\mathcal{B}} \times_{\mathcal{A}} I_{\mathcal{C}}$ with marginals $\overline{p}_{\mathcal{B}}$ and $\overline{p}_{\mathcal{C}}$. Indeed, for fixed $i \in [r]$, the matrix of $\overline{p}^i_{jk}$ for $j \in [s_i],k \in [t_i]$ has rank 1, because $\text{Quad}$ contains all $2 \times 2$-minors for this matrix. Hence the marginals $\overline{p}^i_{j+}$ and $\overline{p}^i_{+k}$ completely determine this matrix.

Finally, we get the maximum likelihood estimate of $I_{\mathcal{B}} \times_{\mathcal{A}} I_{\mathcal{C}}$ by taking  $\overline{p_{\mathcal{B}}}$ and $\overline{p_{\mathcal{C}\,}}$ to be the maximum likelihood estimates of the models $I_{\mathcal{B}}$ and $I_{\mathcal{C}}$.  By Lemma~\ref{lemma:marginaL_{c,i}s_MLE}, the margins $\hat{p}_{\mathcal{B}}$ and $\hat{p}_{\mathcal{C}}$ of the maximum likelihood estimate of the model $I_{\mathcal{B}} \times_{\mathcal{A}} I_{\mathcal{C}}$ are equal to $\overline{p_{\mathcal{B}}}$ and $\overline{p_{\mathcal{C}\,}}$, in particular $\overline{p_{\mathcal{B}}}$ and $\overline{p_{\mathcal{C}\,}}$ are nonnegative. By Proposition~\ref{Birch}, for each toric model there is a unique nonnegative critical point of the likelihood function and it is the maximum likelihood estimate. Hence  $\overline{p_{\mathcal{B}}}$ and $\overline{p_{\mathcal{C}\,}}$ have to be the maximum likelihood estimates for the models $I_{\mathcal{B}}$ and $I_{\mathcal{C}}$.
\end{proof}

Let $\mathcal{T}$ be an $n$-leaf 3-valent tree. We denote the coordinates of a vector $u \in \R^{2^{n-1}}$ by $u_l$ where $l$ corresponds to a labeling of the edges of $\mathcal{T}$. Let $\mathcal{T}'$ be a subtree of $\mathcal{T}$ and  denote the restriction of the labeling $l$ to $\mathcal{T}'$ by $l|_{T'}$. We denote by $u_{T'}$ the marginal sum of $u$ with the edges of $\mathcal{T}$ not in $\mathcal{T}'$ marginalized out, i.e. the vector $u_{T'}$ is indexed by the labelings of $\mathcal{T}'$ and if $l'$ is a labeling of $\mathcal{T}'$ then $(u_{T'})_{l'} = \sum_{l|_{T'}=l'} u_l$. If the subtree is an edge $e$, then the marginal sum is defined in the same way and denoted by $u_e$.  As before, we denote the sum of entries of $u$ by $u_+$.

\begin{corollary} \label{corollary:ML_degree_phylogenetic_models}
For any 3-valent tree $\mathcal{T}$, the ML degree of $V_{\mathcal{T}}$ is one. If $\mathcal{T}$ is tripod and $u$ is a data vector, then the maximum likelihood estimate is
$$
\hat{p}=\frac{u}{u_{+}}.
$$
It $\mathcal{T}$ has more than three leaves, let $\mathcal{T}_1,{\mathcal{T}}_2,\ldots,{\mathcal{T}}_{n-2}$ be the tripods contained in $\mathcal{T}$ and let $e_1,e_2,\ldots,e_{n-3}$ be the inner edges of $\mathcal{T}$. For data vector $u$, the maximum likelihood estimate is
\begin{equation}\label{equation:MLE_phylogenetic_models}
\hat{p}_l = \frac{\prod_{j=1,\ldots,n-2} (\widehat{p_{{\mathcal{T}}_j}})_{l|_{{\mathcal{T}}_j}}}{\prod_{j=1,\ldots,n-3} (\widehat{p_{e_j}})_{l|_{e_j}}},
\end{equation}
where $\widehat{p_{e_j}}$ is the normalized $u_{e_j}$, and $\widehat{p_{{\mathcal{T}}_j}}$ is the maximum likelihood estimate for the tree $\mathcal{T}_j$ and the data vector $u_{{\mathcal{T}}_j}$.
\end{corollary}

The ML degree of a group-based phylogenetic model in probability coordinates is not known to be related to the ML degree in the Fourier coordinates. In particular, the ML degree can be much larger in the probability coordinates than the one in the Fourier coordinates~\cite[Section 6]{Hosten_etal}. In probability coordinates, numerical methods are needed already in the smallest cases to determine the maximum likelihood estimate and the critical points of the likelihood function~\cite{kosta2019maximum}.

\begin{example}
Let $\mathcal{T}$ be the 3-valent four leaf tree, let $\mathcal{T}_1$ and $\mathcal{T}_2$ be the tripods contained in $\mathcal{T}$, and let $e$ be the inner edge of $\mathcal{T}$. We consider the data vector $$u=(u_{00000},u_{11000},u_{00011},u_{11011},u_{10110},u_{10101},
u_{01110},u_{01101})=(17,5,27,5,16,5,19,6)$$
with a total of $100$ observations. Then
\begin{align*}
u_{{\mathcal{T}}_1}&=(u_{000++},u_{110++},u_{101++},u_{011++})=(44,10,21,25),\\
u_{{\mathcal{T}}_2}&=(u_{++000},u_{++110},u_{++101},u_{++011})=(22,35,11,32),\\
u_e&=(u_{++0++},u_{++1++})=(54,46).
\end{align*}
Since the $V_{{\mathcal{T}}_1}=V_{{\mathcal{T}}_2}=\mathbb{P}^3$, we have
$$
\widehat{p_{{\mathcal{T}}_1}}=(\frac{44}{100},\frac{10}{100},\frac{21}{100},\frac{25}{100}), \widehat{p_{{\mathcal{T}}_2}}=(\frac{22}{100},\frac{35}{100},\frac{11}{100},\frac{32}{100}) \text{ and } \widehat{p_e\,}=(\frac{54}{100},\frac{46}{100}).
$$
Then by~(\ref{equation:MLE_phylogenetic_models})
$$
\hat{p}_{00000}=\frac{(\widehat{p_{{\mathcal{T}}_1}})_{000} (\widehat{p_{{\mathcal{T}}_2}})_{000}}{(\widehat{p_e\,})_0}=\frac{44\cdot 22 \cdot 100}{100^2 \cdot 54}=\frac{121}{675}.
$$
Similarly, we can find other coordinates of the maximum likelihood estimate. This gives
$$
\hat{p}=\left(\frac{121}{675},\frac{11}{270},\frac{176}{675},\frac{8}{135},\frac{147}{920},\frac{231}{4600},\frac{35}{184},\frac{11}{184}\right).
$$
We obtain the same result when using Birch's theorem.
\end{example}

Recent work on rational maximum likelihood estimators establishes that a class of tree models known as staged trees have ML degree 1~\cite[Proposition 12]{duarte2019discrete}. In light of Corollary~\ref{corollary:ML_degree_phylogenetic_models}, it is natural to ask if there is any relation between staged tree models and 3-valent phylogenetic tree models. We find that this is the case in the proposition below. In fact, we believe that any codimension zero toric fiber product can be viewed as a generalized staged tree and this is left as a future research direction.  Conversely, Ananiadi and Duarte study when staged trees are codimension-0 toric fiber products in the recent paper~\cite{ananiadi2019gr}. 

Consider a rooted tree $\mathcal{T}$ with at least two edges emanating from every non-leaf vertex of $\mathcal{T}$. Consider a labeling of the edges of $\mathcal{T}$ by the elements of a set $S$. The \textit{floret} associated with a vertex $v$  is the multiset of labels of edges emanating from $v$. The tree $\mathcal{T}$ is called a \textit{staged tree} if any two florets are equal or disjoint. The set of florets is denoted by $F$.

\begin{definition}[Definition 10 in~\cite{duarte2019discrete}] \label{def:stagedtree}
Let $J$ denote the set of all paths from root to leaf in $\mathcal{T}$. For a path $j \in J$ and a label $s \in S$, let $\mu_{sj}$ denote the number of times an edge labeled by $s$ appears on the path $j$. A \textit{staged tree model} is the image of the parameter space
$$
\Theta = \left\{ \left(\theta_s\right)_{s \in S} \in \left(0,1\right)^S:\sum_{s \in f} \theta_s=1 \text{ for all } f\in F \right\}
$$
under the map $p_j = \prod \theta_s ^{\mu_{sj}}$.
\end{definition}

\begin{proposition} \label{prop:stagedphyl}
All 3-valent phylogenetic tree models as defined in Definition~\ref{def:phylogenetic_models} are staged tree models.
\end{proposition}

\begin{proof}
The staged tree begins with a tripod that can be chosen arbitrarily. The first stage has 4 edges corresponding to the four labelings of this tripod. The tripod corresponding to any subsequent stage must share an edge with a tripod corresponding to a previous stage. The florets are binary and correspond to the two possible labelings of the common edge. The parameters $\theta_s$ for a given stage are marginal probabilities for the tripod divided by the marginal probabilities for the edge shared with a tripod corresponding to a previous stage. In this way, the staged tree corresponding to a phylogenetic model on a 3-valent $n$-leaf tree has one stage of edges for every tripod in the $n$-leaf tree.
\end{proof}

\begin{example}
The staged trees corresponding to the phylogenetic models on the tripod and the 3-valent 4-leaf tree are depicted in Figures~\ref{figure:staged_tree1} and~\ref{figure:staged_tree2}. The vertices filled with the same color have the same florets. Unfilled vertices have all different florets. In Figure~\ref{figure:staged_tree2}, the parameters $\theta_i$ are equal to
\begin{align*}
\theta_1 &= p_{000++}, &\theta_2 &= p_{110++}, &\theta_3 &= p_{101++}, &\theta_4 &= p_{011++},\\
\theta_5 &= \frac{p_{++000}}{p_{++0++}}, &\theta_6 &= \frac{p_{++011}}{p_{++0++}}, &\theta_7 &= \frac{p_{++110}}{p_{++1++}}, &\theta_8 &=\frac{p_{++101}}{p_{++1++}}.
\end{align*}

\begin{figure}
\centering
\begin{minipage}{.5\textwidth}
  \centering
  \includegraphics[height=5cm]{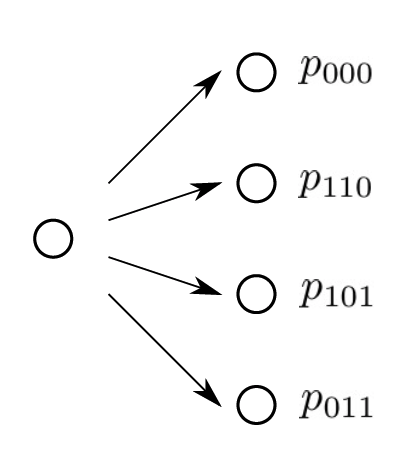}
  \captionof{figure}{Staged tree for the tripod}
  \label{figure:staged_tree1}
\end{minipage}%
\begin{minipage}{.5\textwidth}
  \centering
  \includegraphics[height=5cm]{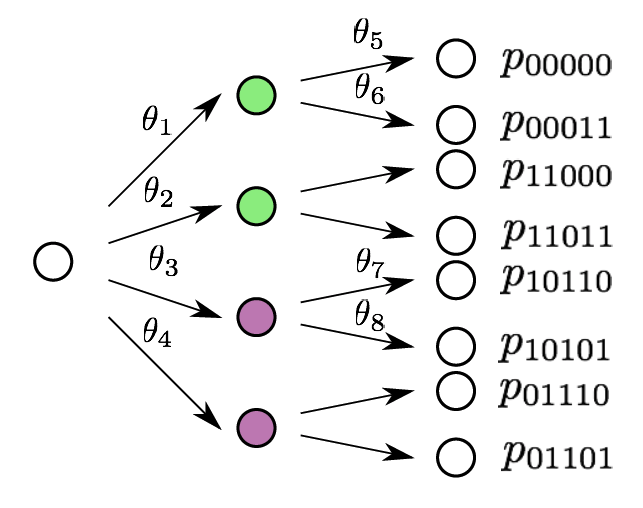}
  \captionof{figure}{Staged tree for the $4$-leaf tree}
  \label{figure:staged_tree2}
\end{minipage}
\end{figure}

\end{example}

Next, we study the ML degree drop for small phylogenetic models. One can see from Table~\ref{table:phylogenetic_models} that if $\mathcal{T}$ is a 3-valent tree with at least four leaves then the degree of the phylogenetic variety is strictly larger than the sum of degrees of the components of the intersection $\mathcal{L}'_{c,u}$ of the phylogenetic variety with the affine space defined in~(\ref{equations:linear_subspace}). This implies that this intersection is not generically transverse. Since $c=(1,1,\dots,1)$ in this example, we drop the subscript from $L_c$ and $L_{c,i}$.

In the case of the $4$-leaf tree, the intersection of $V_{\mathcal{T}}$ and the linear subspace defined by $L(p)=L_1(p)=\ldots=L_5(p)=0$ has one component of dimension $1$ and degree $1$. It is the same component as in Table~\ref{table:phylogenetic_models} that does not contribute to the ML degree. Since the intersection is not generically transverse, the degree $1$ of this component gives only a  lower bound to the difference $\deg V_{\mathcal{T}} - \mldeg V_{\mathcal{T}}=3$.

In the case of the $5$-leaf tree, the intersection of $V_{\mathcal{T}}$ and the linear subspace defined by $L(p)=L_1(p)=\ldots=L_7(p)=0$ has three components each of dimension $3$ and degrees $1,3,3$. These components are the three components in Table~\ref{table:phylogenetic_models} that do not contribute to the ML degree of $V_{\mathcal{T}}$. Since the intersection is not generically transverse, the sum $1+3+3=7$ of degrees of the components gives only a lower bound to the difference $\deg V_{\mathcal{T}} - \mldeg V_{\mathcal{T}}=33$.

In both cases, the intersection is transverse only at the zero-dimensional component of degree 1 that gives the MLE.

\begin{table}[ht]
\centering
\begin{tabular}{|c|c|c|c|c|c|c|c|}
\hline
tree &  dim & deg & dim int. & deg int. & \# comp. & (dim, deg) comp. & $\sum$deg comp. \\
\hline \hline
 tripod & 3 & 1 & 0 & 1 & 1 & (0,1) & 1 \\
 \hline
 4-leaf & 5 & 4 & 1 & 1 & 2 & (0,1),(1,1) & 2\\
 \hline
 5-leaf & 7 & 34 & 3 & 7 & 4 & (0,1),(3,1),(3,3),(3,3) & 8\\
 \hline
\end{tabular}
\caption{Properties of phylogenetic ideals on 3-valent trees. The data presented in this table are: the dimension of the variety $V_{\mathcal{T}}$; the degree of the variety $V_{\mathcal{T}}$; the dimension of the intersection of $V_{\mathcal{T}}$ and the linear subspace defined by $L(p)=L_1(p)=\ldots=L_{d-1}(p)=0$; the degree of the intersection of $V_{\mathcal{T}}$ and the linear subspace defined by $L(p)=L_1(p)=\ldots=L_{d-1}(p)=0$; the number of irreducible components of this intersection; the dimension and the degree of each such irreducible component in the form (dim, deg); the sum of the degrees of the irreducible components.}
\label{table:phylogenetic_models}
\end{table}

Huh~\cite{Huh} proved that if the MLE of a statistical model is a rational function of the data, then it has to be an alternating product of linear forms of specific form. In particular, the MLE is given by
$$
\hat{p}_j = \lambda_j \prod_{i=1}^m (\sum_{k=1}^{n+1} h_{ik} u_k)^{h_{ij}} \text{ ,}
$$
where $\lambda=(\lambda_0,\lambda_1,\ldots,\lambda_n)$ and $H$ is an $m \times (n+1)$ integer matrix whose columns sum to zero. Such a map is called Horn uniformization and the matrix $H$ is called a Horn matrix. Duarte, Marigliano and Sturmfels~\cite[Theorem 1]{duarte2019discrete} proved that then there exists a determinantal triple $(A,\Delta,\bf{m})$ such that the statistical model is the image under the monomial map $\phi_{\Delta,\bf{m}}$ of an orthant of the dual of the toric variety $Y_{A}$.

\begin{example} \label{example:phylogenetic_Huh}
Let $\mathcal{T}$ be the $3$-valent $4$-leaf tree. Then $\lambda=(1,1,\ldots,1)$ and
$$
H = \kbordermatrix{
& p_{00000} & p_{00011} & p_{11000} & p_{11011} & p_{10110}& p_{10101} &
p_{01110} & p_{01101}\\
000++ & 1 & 1 & 0 & 0 & 0 & 0 & 0 & 0 \\
110++ & 0 & 0 & 1 & 1 & 0 & 0 & 0 & 0 \\
101++ & 0 & 0 & 0 & 0 & 1 & 1 & 0 & 0 \\
011++ & 0 & 0 & 0 & 0 & 0 & 0 & 1 & 1 \\
+++++ & -1 & -1 & -1 & -1 & -1 & -1 & -1 & -1\\
++000 & 1 & 0 & 1 & 0 & 0 & 0 & 0 & 0 \\
++011 & 0 & 1 & 0 & 1 & 0 & 0 & 0 & 0 \\
++110 & 0 & 0 & 0 & 0 & 1 & 0 & 1 & 0 \\
++101 & 0 & 0 & 0 & 0 & 0 & 1 & 0 & 1 \\
++0++ & -1 & -1 & -1 & -1 & 0 & 0 & 0 & 0 \\
++1++ & 0 & 0 & 0 & 0 & -1 & -1 & -1 & -1
}.
$$
If $\mathcal{T}$ is a $3$-valent $n$-leaf tree, then $H$ has $6(n-2)-1$ rows and $2^{n-1}$ columns. Each column contains $n-2$ ones and the same number of minus ones, all other entries are zeros. The vector $\lambda$ has all its entries equal to $(-1)^{n-2}$.

The rows of the matrix
\setcounter{MaxMatrixCols}{11}
$$
A=\begin{bmatrix}
1 & 1 & 1 & 1 & 1 & 0 & 0 & 0 & 0 & 0 & 0 \\
0 & 0 & 0 & 0 & 0 & 1 & 1 & 0 & 0 & 1 & 0\\
0 & 0 & 0 & 0 & 0 & 0 & 0 & 1 & 1 & 0 & 1\\
2 & 2 & 1 & 1 & 1 & 0 & 0 & 1 & 1 & 1 & 1\\
1 & 1 & 2 & 2 & 1 & 1 & 1 & 0 & 0 & 1 & 1
\end{bmatrix}
$$
give a basis of the left kernel of the Horn matrix $H$. The discriminant
$$
\Delta_A =  -x_5x_{10}x_{11}+x_1x_6x_{11}+x_1x_7x_{11}+x_2x_6x_{11}+x_2x_7x_{11}+x_3x_8x_{10}+x_3x_9x_{10}+x_4x_8x_{10}+x_4x_9x_{10}
$$
vanishes on the dual of the toric variety
\begin{scriptsize}
$$
Y_A =\overline{\left\{ \left(t_1 t_4^2 t_5 : t_1 t_4^2 t_5 : t_1 t_4 t_5^2 : t_1 t_4 t_5^2 : t_1 t_4 t_5 : t_2 t_5 : t_2 t_5 : t_3 t_4 : t_3 t_4 : t_2 t_4 t_5:t_3 t_4 t_5 \right) \in \R \mathbb{P}^3: t_1,t_2,t_3,t_4,t_5 \in \R^* \right\}}.
$$
\end{scriptsize}
The toric variety $Y_A$ is of dimension $4$ and degree $4$.
Let $\textbf{m}=-x_5 x_{10} x_{11}$. Then
$$
\frac{1}{\textbf{m}} \Delta_A = 1 - \frac{x_1 x_6}{x_5 x_{10}} - \frac{x_1 x_7}{x_5 x_{10}} - \frac{x_2 x_6}{x_5 x_{10}} - \frac{x_2 x_7}{x_5 x_{10}} - \frac{x_3 x_8}{x_5 x_{11}} - \frac{x_3 x_9}{x_5 x_{11}} - \frac{x_4 x_8}{x_5 x_{11}} -\frac{x_4 x_9}{x_5 x_{11}}.
$$
This gives the monomial map
$$
\phi_{(\Delta_A,\textbf{m})} = \left( \frac{x_1 x_6}{x_5 x_{10}} , \frac{x_1 x_7}{x_5 x_{10}} , \frac{x_2 x_6}{x_5 x_{10}} , \frac{x_2 x_7}{x_5 x_{10}} , \frac{x_3 x_8}{x_5 x_{11}} , \frac{x_3 x_9}{x_5 x_{11}} , \frac{x_4 x_8}{x_5 x_{11}} ,\frac{x_4 x_9}{x_5 x_{11}} \right).
$$
The model $\sM$ is the image of
$$
Y^*_{A,\sigma} = \left\{ x \in Y^*_A: x_1,x_2,x_3,x_4,x_6,x_7,x_8,x_9>0, x_5,x_{10},x_{11}<0 \right\}
$$
under the map $\phi_{(\Delta_A,\textbf{m})}$. The maximum likelihood estimator is given by $\phi_{(\Delta_A,\textbf{m})} \circ H:\Delta_7 \rightarrow \sM$.
\end{example}

If the `3-valent' hypothesis is dropped, the ML degree does not need to be one. We conclude with an example of a toric fiber product where the ML degree is greater than one.

\begin{example}
Consider the tree $\mathcal{T}$ with five leaves that has two inner nodes of degrees three and four. Then $\mathcal{T}$ is the union of a tripod $\mathcal{T}_1$ and a four-leaf claw tree $\mathcal{T}_2$ with two edges identified. The ideal $I_{\mathcal{T}}$ is a toric fiber product of $I_{{\mathcal{T}}_1}$ and $I_{{\mathcal{T}}_2}$. The ML degree of $I_{{\mathcal{T}}_1}$ is $1$ by Corollary~\ref{corollary:ML_degree_phylogenetic_models}. The ideal of $I_{{\mathcal{T}}_2}$ is generated by
$$
x_{1001}x_{0110}-x_{0000}x_{1111},x_{0101}x_{1010}-x_{0000}x_{1111},x_{1100}x_{0011}-x_{0000}x_{1111}.
$$
It is an ideal of codimension $3$ and degree $8$. Its maximum likelihood degree is $5$. Hence also the ML degree of $I_{\mathcal{T}}$ is $5$ by Theorem~\ref{theorem:ML_degree_of_toric_fiber_products}.

Similarly, if $\mathcal{T}$ is the six-leaf tree with two inner nodes of degree four then the ML degree of $I_{\mathcal{T}}$ is $25$.
\end{example}

\noindent

\bibliography{bibl}
\bibliographystyle{plain}

\bigskip \bigskip

\noindent
\footnotesize {\bf Authors' addresses:}

\begin{multicols}{3}
\smallskip
\noindent
Carlos Am\'endola\\
Department of Mathematics,\\
Technische Universit\"at M\"unchen,\\
 Munich 85748, Germany.\\
\texttt{carlos.amendola@tum.de}
\vfill\null
\columnbreak

\noindent
Dimitra Kosta\\
School of Mathematics\\ and Statistics,\\
University of Glasgow,\\
Glasgow G12 8SQ, UK.\\
\texttt{Dimitra.Kosta@glasgow.ac.uk}
\vfill\null
\columnbreak

\noindent
Kaie Kubjas\\
Department of Mathematics and Systems Analysis,\\
Aalto University,\\
FI-00076 Aalto, Finland.\\
\texttt{kaie.kubjas@aalto.fi}

\end{multicols}

\end{document}